\documentclass[11pt]{amsart}

\usepackage{amsmath,amsfonts,amssymb,amsthm,mathtools}

\usepackage{verbatim}
\usepackage{amsmath,amsfonts}
\usepackage[mathscr]{euscript} 
\usepackage{amsthm}
\usepackage{blkarray,bigstrut} 
\usepackage{url}

\usepackage{graphicx} 

\usepackage{enumitem} 

\usepackage{hyperref} 
\hypersetup{colorlinks} 

\usepackage[colorinlistoftodos]{todonotes}

\RequirePackage{cleveref}
\usepackage{hypcap}
\hypersetup{colorlinks=true, citecolor=darkblue, linkcolor=darkblue}
\definecolor{darkblue}{rgb}{0.0,0,0.7}
\newcommand{\darkblue}{\color{darkblue}}

\definecolor{darkred}{rgb}{0.68,0,0}
\newcommand{\darkred}{\color{darkred}}

\definecolor{darkgreen}{rgb}{0,.38,0}
\newcommand{\darkgreen}{\color{darkgreen}}

\newcommand{\defn}[1]{\emph{\darkblue #1}}
\newcommand{\defna}[1]{\emph{\darkred #1}}
\newcommand{\defnb}[1]{\emph{\darkblue #1}}
\newcommand{\defng}[1]{\emph{\darkgreen #1}}

\setlist[enumerate]{
	label=\textnormal{({\roman*})},
	ref={\roman*}}

\makeatletter
\def\th@plain{%
	\thm@notefont{}
	\itshape 
}
\def\th@definition{%
	\thm@notefont{}
	\normalfont 
}
\makeatother

\makeatletter
\def\fdsy@scale{1}
\newcommand\fdsy@mweight@normal{Book}
\newcommand\fdsy@mweight@small{Book}
\newcommand\fdsy@bweight@normal{Medium}
\newcommand\fdsy@bweight@small{Medium}

\DeclareFontFamily{U}{FdSymbolB}{}
\DeclareFontShape{U}{FdSymbolB}{m}{n}{
	<-7.1> s * [\fdsy@scale] FdSymbolB-\fdsy@mweight@small
	<7.1-> s * [\fdsy@scale] FdSymbolB-\fdsy@mweight@normal
}{}
\DeclareFontShape{U}{FdSymbolB}{b}{n}{
	<-7.1> s * [\fdsy@scale] FdSymbolB-\fdsy@bweight@small
	<7.1-> s * [\fdsy@scale] FdSymbolB-\fdsy@bweight@normal
}{}

\makeatother

\DeclareSymbolFont{fdrelations}{U}{FdSymbolB}{m}{n}
\SetSymbolFont{fdrelations}{bold}{U}{FdSymbolB}{b}{n}

\DeclareMathSymbol{\lescc}{\mathrel}{fdrelations}{66}

\newtheorem{thm}{Theorem}[section]
\newtheorem{lemma}[thm]{Lemma}

\newtheorem{cor}[thm]{Corollary}
\newtheorem{prop}[thm]{Proposition}
\newtheorem{conj}[thm]{Conjecture}
\newtheorem{conv}[thm]{Convention}

\theoremstyle{definition}
\newtheorem{ex}[thm]{Example}

\newtheorem{rem}[thm]{Remark}

\numberwithin{figure}{section}
\numberwithin{equation}{section}

\def\bd{\triangledown}

\def\nn{\mathbb N}

\def\rr{\mathbb R}

\def\sm{\smallsetminus}

\def\la{\lambda}
\def\ga{\gamma}
\def\si{\sigma}

\def\ep{\varepsilon}
\def\al{\alpha}
\def\be{\beta}
\def\om{\omega}

\def\cB{\mathcal B}

\def\cF{\mathcal F}

\def\cM{\mathcal{M}}

\def\cP{\mathcal P}

\def\ssu{\subset}

\def\<{\langle}
\def\>{\rangle}

\def\rk{\textnormal{rk}}

\def\vt{\vartriangleleft}
\def\vte{\trianglelefteq}

\def\0{{\mathbf 0}}
\def\id{{\rm id}}

\def\.{\hskip.06cm}
\def\ts{\hskip.03cm}

\def\lra{\leftrightarrow}

\def\di{{\small{\ts\diamond\ts}}}

\def\ze{{\zeta}}

\newcommand{\sort}{\mathrm{sort}}

\newcommand{\SYT}{\operatorname{{\rm SYT}}}

\def\nin{\noindent}

\def\SP{{\textsc{\#P}}}

\def\SP{{\textsc{\#{}P}}}

\def\Pb{{\mathbb{P}}}
\def\Eb{{\mathbb{E}}}
\def\Varb{{\text{\sc Var}}}
\def\Covb{{\text{\sc Cov}}}

\def\aF{\textrm{F}}

\def\ah{\textrm{h}}

\def\aM{\textrm{M}}

\def\aMr{\textrm{\em M}}
\def\aN{\textrm{N}}
\def\aNr{\textrm{\em N}}

\def\aQ{\textrm{Q}}

\def\av{\textrm{v}}

\def\bM{\textbf{\textrm{M}}\hskip-0.03cm{}}

\def\bMr{\textbf{\textrm{\em M}}\hskip-0.03cm{}}

\DeclareMathOperator{\Bc}{\mathcal{B}}

\DeclareMathOperator{\Cc}{\mathcal{C}}

\DeclareMathOperator{\Cov}{\textnormal{Cov}}
\DeclareMathOperator{\Dc}{\mathcal{D}}

\DeclareMathOperator{\Ec}{\mathcal{E}} 

\DeclareMathOperator{\Rb}{\mathbb{R}}

\newcommand{\mn}{{{\text{min}}}} 

\DeclareMathOperator{\hb}{\mathbf{h}}

\DeclareMathOperator{\vb}{\mathbf{v}}

\DeclareMathOperator{\xb}{\mathbf{x}}

\DeclareMathOperator{\yb}{\mathbf{y}}
\DeclareMathOperator{\zb}{\mathbf{z}}

 \def\precc{\preccurlyeq} 

\newcommand{\down}{\textnormal{down}}
\newcommand{\up}{\textnormal{up}}

\newcommand{\pa}{a} 

\newcommand{\Mf}{\mathscr{M}}

\title{Correlation inequalities for linear extensions}

\date{\today}

\author[Swee Hong Chan \and Igor Pak]{\ Swee Hong Chan$^{\star}$ \ \. \and \ \. Igor~Pak$^{\di}$}

\thanks{\thinspace ${\hspace{-.45ex}}^\star$Department of Mathematics,
Rutgers University, Piscataway, NJ, 08854.
 \.  Email: \ts \texttt{sweehong.chan@rutgers.edu}}

\thanks{\thinspace ${\hspace{-.66ex}}^\di$Department of Mathematics,
UCLA, Los Angeles, CA, 90095. \.  Email: \ts \texttt{{pak}@math.ucla.edu}}



\begin{document}

\begin{abstract}
We employ the combinatorial atlas technology to prove new correlation
inequalities for the number of linear extensions of finite posets.
These include the approximate independence of probabilities and
expectations of values of random linear extensions, closely related
to Stanley's inequality.  We also give applications to the numbers
of standard Young tableaux and to Euler numbers.
\end{abstract}

	\maketitle

\section{Introduction} \label{s:intro}

If you really want to hear about it, the first thing you'll probably want to know
about \emph{Poset Theory} \ts is that it was born and languished for decades without
any tools whatsoever.  It is still a question whether the area has caught up with
the other, more established parts of Combinatorics, but by now it is definitely
in possession of powerful tools of various nature, which makes it at least somewhat
prominent, if not prosperous.

In this paper we obtain a melange of new correlation inequalities
for the number of linear extensions of finite posets
using the \defna{combinatorial atlas} \ts technology \cite{CP1,CP2}.
This is the tool we recently developed and whose power has yet to be
fully exploited.
Viewed individually, each new inequality shows that there is more to the story.
Taken together, these inequalities offer a glimpse at the big world of poset
inequalities that we are just beginning to understand.

\smallskip

\subsection{Linear extensions}\label{ss:intro-LE}
Let \ts $P=(X,\prec)$ \ts be a \defnb{partially ordered set}
on the ground set $X$ of size $|X|=n$, and with the partial order~``$\prec$''.
A \defn{linear extension} of $P$ is a bijection \ts $f: X \to \{1,\ldots,n\}$ \ts
that is order-preserving: \ts $x \prec y$ \ts implies \ts $f(x)<f(y)$, for all \ts
$x,y\in X$.
We denote by \ts $\Ec(P)$ \ts the set of linear extensions of $P$, and by \ts
$e(P):=|\Ec(P)|$ \ts the number of linear extensions.

An element \ts $x\in X$ \ts is \defn{minimal} \ts if there is no \ts
$y\in X$ \ts such that $y \prec x$.
Denote by \ts $\min(P) \subseteq X$ \ts the subset of minimal elements in~$P$.
We use \ts $C_n$ \ts and \ts $A_n$ \ts to denote the $n$-element \defn{chain} \ts
and \defn{antichain}, respectively.
For an element \ts $x \in X$, we denote by \ts $P-x$ \ts the induced subposet of~$P$
on the subset \ts $X-x$. See Section~\ref{s:notation} for further poset notations.

\smallskip

\begin{thm}[{\rm cf.\ Theorem~\ref{t:poset-corr-deletion-strong}}{}]
\label{t:poset-corr-deletion}
Let \ts $P=(X,\prec)$ \ts be a poset with \ts $|X|=n> 2$ \ts elements.
Let \ts $x,y\in \min(P)$ \ts  be distinct minimal elements of~$P$. Then:
\begin{equation}\label{eq:poset-corr-tight}
  \frac{n}{n-1} \ \leq \ \frac{e(P) \.\cdot\. e(P-x-y)}{e(P-x) \.\cdot\. e(P-y)}  \ \leq \ 2 \ts .
\end{equation}
\end{thm}

\smallskip

This correlation inequality is the most natural
and the simplest to state.
The lower bound in~\eqref{eq:poset-corr-tight} is a special case of the
\defng{Fishburn inequality} (see $\S$\ref{ss:finrem-Fish}), while the upper
bound is new and is a special case of the upper bound in
Theorem~\ref{t:poset-corr-deletion-strong}.
%
Note that the lower bound is tight for \ts $P=A_n$ \ts and
the upper bounds is tight for the linear sum \ts $P=A_2\oplus C_{n-2}$\..


\smallskip

Correlation inequalities are best understood in probabilistic notation.
Let \ts $\Pb$ \ts and \ts $\Eb$ \ts be taken over the uniform random
linear extension \ts $f\in\Ec(P)$.
The inequality~\eqref{eq:poset-corr-tight} can be rewritten as
\begin{equation}\label{eq:poset-corr-tight-probab}
  \frac{n}{n-1} \ \leq \ \frac{\Pb[f(x)=1, f(y)=2]}{\Pb[f(x)=1] \.\cdot\. \Pb[f(y)=1]}  \ \leq \ 2.
\end{equation}

\smallskip

The following theorem gives a similar upper bound for the covariances:

\smallskip

\begin{thm}[{\rm see~$\S$\ref{ss:proof-poset-covariances}}{}]\label{t:poset-covariances}
Let \ts $P=(X,\prec)$ \ts be a finite poset, and let \ts $x,y \in X$ \ts
be fixed poset elements. Then:
\begin{equation}\label{eq:poset-cov}
\frac{\Eb[f(x) \ts f(y)]  \. + \. \Eb\big[\min\{f(x),f(y)\}\big]}{\Eb[f(x)] \.\cdot\. \Eb[f(y)]} \ \le \, 2 \ts.
\end{equation}
\end{thm}

\smallskip
We note that \eqref{eq:poset-cov} is asymptotically tight; see Example~\ref{ex:poset-AK}.
Also
note that it follows from \eqref{eq:poset-cov} that
\begin{align}\label{eq:poset-cov-2}
	\frac{\Covb(f(x), f(y))}{\Eb[f(x)] \.\cdot\. \Eb[f(y)]}  \ = \ \frac{\Eb[f(x) \ts f(y)]}{\Eb[f(x)] \.\cdot\. \Eb[f(y)]} \. - \. 1  \ \leq \  1.
\end{align}
See also~$\S$\ref{ss:app-second} for an application to the second moment estimates.

\medskip

\subsection{Stanley inequality}\label{ss:intro-Stanley}
For linear extensions, one can use both elements and values to set up
correlation inequalities.  The following celebrated result by Stanley
is foundational in the area:

\smallskip

\begin{thm}[{\rm \defn{Stanley inequality}~{\cite[Thm~3.1]{Sta}}}{}]\label{t:poset-Stanley}
Let \ts $P=(X,\prec)$ \ts be a poset with \ts $|X|=n$ \ts elements. Fix \ts
$x\in X$ \ts and let \. $2 \leq k \leq n-1$. Denote by \. $e_k(P)$ \.
the number of linear extensions \ts $f\in \Ec(P)$ \ts such that \. $f(x)=k$.
Then:
\begin{equation}\label{eq:Stanley-thm}
 e_k(P)^2 \ \geq \  e_{k-1}(P) \. \cdot \. e_{k+1}(P).
\end{equation}
\end{thm}

\smallskip

In probabilistic notation, Stanley's inequality can be restated as
\begin{equation}\label{eq:Stanley}
\Pb[f(x)=k]^2 \  \geq \  \Pb[f(x)=k-1] \. \cdot \. \Pb[f(x)=k+1]\ts.
\end{equation}

\smallskip

\nin
%
The proof in \cite{Sta} used the \defng{Alexandrov--Fenchel
inequality} \ts applied to order polytopes, see $\S$\ref{ss:finrem-Stanley}.
Stanley's inequality was famously extended by Kahn and Saks~\cite{KS}, to a prove a weak version of the
\. \defng{$\frac{1}{3}$--$\frac{2}{3}$ \. Conjecture}, see~$\S$\ref{ss:finrem-1323}.

Stanley inequality~\eqref{eq:Stanley} has remained mysterious until recently,
when the equality conditions have  been established by Shenfeld and van~Handel~\cite{SvH}
by advancing geometric arguments.  In our paper~\cite{CP1}, we gave a linear
algebraic proof of Stanley inequality, the equality conditions, and the
generalization to weighted linear extensions, see also~$\S$\ref{ss:finrem-Stanley}.

\smallskip

Let us single out two slightly
nonstandard general consequences of Stanley inequality:

\smallskip

\begin{cor}[{\rm see $\S$\ref{ss:cor-proof-two-lemmas}}{}]\label{c:poset-Stanley-cdf}
Let \ts $P=(X,\prec)$ \ts be a poset with \ts $|X|=n$ \ts elements. Fix \ts
$x\in X$ \ts and let \. $1 \leq k \leq n-1$. Then:
    \begin{equation}\label{eq:Stanley-cor1}
		  \Pb[f(x)> k]^2 \ \geq \ \Pb[f(x) > k-1] \. \cdot \. \Pb[f(x) > k+1].
	\end{equation}
and
    \begin{equation}\label{eq:Stanley-cor2}
\Pb[f(x)  =  1]  \. \cdot \. \Pb[f(x)>1]  \  \leq \ \Pb[ f(x)  =  2]\ts.
	\end{equation}
\end{cor}

\smallskip

For \ts $k=1$, inequality \eqref{eq:Stanley-cor1} further simplifies to
    \begin{equation}\label{eq:Stanley-cor11}
	 \Pb[f(x)> 1]^2 \ \geq \	\Pb[f(x) > 2]  .
	\end{equation}

\smallskip

%
We propose the following unusual extension of Stanley inequality to general subsets of
the ground set. Fix a nonempty subset $A\subseteq X$.  For a linear extension
\ts $f\in \Ec(P)$, define
\[
f(A) \, := \, \big\{\ts f(x) \, : \, x \in A \ts\big\} \quad \text{and} \quad
f_\mn(A) \, := \, \min \ts f(A).
\]
Note that \ts $f_\mn(A)=f(x)$ \ts for all singletons \ts $A=\{x\}$, where \ts $x\in X$.

\smallskip

\begin{conj}[{\rm\defn{Extended Stanley inequality}\ts}{}]\label{conj:poset-Stanley-subset}
Let \ts $P=(X,\prec)$ \ts be a poset with \ts $|X|=n$ \ts elements. Fix a nonempty
subset \ts
$A \subseteq  X$, and let \. $2 \leq k \leq n-1$. Then:
\begin{equation}\label{eq:Stanley-conj}
     \Pb[f_{\min}(A)=k]^2   \ \geq \  \Pb[f_{\min}(A)=k-1]  \. \cdot \. \Pb[f_{\min}(A)=k+1]\ts.
\end{equation}
\end{conj}

\smallskip

To justify the conjecture, we prove that the inequalities \eqref{eq:Stanley-cor2} and
\eqref{eq:Stanley-cor11} extend to all subsets (see also~$\S$\ref{ss:finrem-conj}).

\smallskip

\begin{thm}[{\rm see $\S$\ref{ss:proof-poset-Stanley-subset}}{}]
	\label{t:poset-Stanley-subset}
Let \ts $P=(X,\prec)$ \ts be a poset on \ts $|X|\ge 2$ \ts elements. Fix a nonempty
subset \ts $A \subseteq  X$.  Then:
	\begin{equation}\label{eq:Stanley-ext1}
		    \Pb[f_{\min}(A)> 1]^2 \ \geq \  \Pb[f_{\min}(A) > 2].
	\end{equation}
and
\begin{equation}\label{eq:Stanley-ext2}
    \Pb[f_{\min}(A)  =  1]  \.\cdot \. \Pb[f_{\min}(A)>1]  \  \leq \   \Pb[f_{\min}(A) =   2]\ts.
	\end{equation}
\end{thm}

\smallskip

Although Conjecture~\ref{conj:poset-Stanley-subset} is written as a natural generalization
of Stanley Theorem~\ref{t:poset-Stanley}, the inequalities \eqref{eq:Stanley-ext1}
and \eqref{eq:Stanley-ext2} are best understood as
\begin{equation}\label{eq:Stanley-subset1}
 \Pb[1\notin f(A)]^2 \ \geq \ \Pb[1,2\notin f(A)]  \qquad \text{and}
\end{equation}
\begin{equation}\label{eq:Stanley-subset2}
\Pb[1\in f(A)]  \.\cdot \. \Pb[1\notin f(A)] \ \leq \ \Pb[1\notin f(A), 2\in f(A)]\ts,
\end{equation}
respectively.
We show in $\S$\ref{ss:cor-proof-two-lemmas}, that Conjecture~\ref{conj:poset-Stanley-subset}
implies Theorem~\ref{t:poset-Stanley-subset}.
The following corollary is an immediate consequence of \eqref{eq:Stanley-subset1} and \eqref{eq:Stanley-subset2}.

\smallskip

\begin{cor}[{\rm cf.\ Lemma~\ref{lem:poset-Stanley-subset-cdf-Nform}}{}] \label{c:Stanley-poset-two}
Let \ts $P=(X,\prec)$ \ts be a poset on \ts $|X|\ge 2$ \ts elements, and let \ts $A \subseteq X$ \ts be
a nonempty subset of elements.  Then:
\begin{equation}\label{eq:Stanley-ext3}
 \Pb[1\in   f(A), 2\notin f(A)]^2  \ \geq \ \Pb[1,2\in f(A)] \. \cdot \. \Pb[1,2\notin f(A)]\ts.
\end{equation}
\end{cor}

\begin{proof}  Multiply \eqref{eq:Stanley-subset1} for subsets $A$ and $X\sm A$. Then use \eqref{eq:Stanley-subset2}.
\end{proof}

\medskip

\subsection{Multiple subsets}\label{ss:intro-Stanley-more}
We now consider other correlation inequalities obtained by replacing poset
elements with subsets.  First, we have the following
natural generalization of Theorem~\ref{t:poset-covariances}.

\smallskip

\begin{thm}[{\rm see~$\S$\ref{ss:proof-poset-covariances}}{}]\label{t:poset-covariances-multiple}
Let \ts $P=(X,\prec)$ \ts be a finite poset, and let  \ts $A,B\subseteq X$ \ts be nonempty
subsets. Then:
	\begin{equation}\label{eq:poset-covariances-multiple}
		\frac{\Eb\big[f_{\min}(A) \ts  f_{\min}(B) \big] \. + \. \Eb\big[f_{\min}(A\cup B)\big]}{\Eb\big[f_{\min}(A)\big]  \.\cdot\.  \Eb\big[f_{\min}(B)\big]} \ \leq \ 2\ts.
	\end{equation}
\end{thm}

\smallskip

Let us emphasize that here \ts $A$ \ts and \ts $B$ \ts are \emph{arbitrary} \ts subsets of the ground set~$X$.
Next, we have the following variation of Corollary~\ref{c:Stanley-poset-two}
for two subsets of minimal elements.

\smallskip

For an element \ts $b\in X$, denote by \. $b  \!\uparrow \. := \ts \{x\in X \. : \. x\succcurlyeq b\}$ \.
the \defn{upper order ideal} generated by~$b$.  For
a subset \ts $B\subseteq X$, denote by \ts $B\!\uparrow \. := \ts \cup_{b\in B} \, b\!\uparrow$ \. the
\defn{upper closure} of~$B$.
\smallskip

\begin{thm}[{\rm see~$\S$\ref{ss:proof-poset-disjoint}}{}]
\label{t:poset-disjoint-logconcave}
Let \ts $P=(X,\prec)$ \ts be a finite poset, and let \ts $A,B \subset \min(P)$ \ts be
disjoint nonempty subsets of minimal elements.  Then:
	\begin{equation}\label{eq:poset-disjoint-logconcave-new}
    \aligned
	& \Pb\big[1\in f(A), 2 \in f(B)\big]^2 \ \geq \ \Pb\big[1\in f(A), \ts 2 \in f(A\!\uparrow)\ts\big] \. \cdot \. \Pb\big[1\in f(B), \ts 2 \in f(B\!\uparrow)\big]\..
    \endaligned
	\end{equation}
\end{thm}

\smallskip

We conclude with the following  three-subset variation:

\smallskip

\begin{thm}[{\rm see~$\S$\ref{ss:proof-poset-disjoint}}{}]
	\label{t:poset-disjoint-half-correlation}
Let \ts $P=(X,\prec)$ \ts be a finite poset, and let \ts $A,B,C \subset \min(P)$ \ts be
disjoint nonempty subsets of minimal elements. Then:
\begin{equation}\label{eq:Stanley-ext-12}
\frac{\Pb\big[1\in f(C), \ts 2 \in f(C\!\uparrow)\big] \. \cdot \. \Pb\big[1\in f(A), \ts 2\in f(B)\big]}{\Pb\big[1\in f(A), \ts 2\in f(C)\big] \.\cdot \. \Pb\big[1\in f(B), \ts 2\in f(C)\big]}
		\ \leq \  2\ts.
\end{equation}
Moreover, there exists a permutation \. $(A',B',C')$ \. of \. $(A,B,C)$ \. such that
\begin{equation}\label{eq:Stanley-ext-12-prime}
\frac{\Pb\big[1\in f(C'), \ts 2 \in f(C'\!\uparrow)\big] \. \cdot \. \Pb\big[1\in f(A'), \ts 2\in f(B')\big]}{\Pb\big[1\in f(A'), \ts 2\in f(C')\big] \.\cdot \. \Pb\big[1\in f(B'), \ts 2\in f(C')\big]} \
\leq \ 1\ts.
\end{equation}
\end{thm}

%
%
%

\medskip

\subsection{Background and historical remarks}\label{ss:intro-back}
In Combinatorics, correlation inequalities go back to the work of
Kirchhoff on electrical networks.  In modern notation, his result can
be reformulated as saying that for all edges \ts $e,e'\in E$, we have:
\begin{equation}\label{eq:Kir}
\Pb(e,e'\in T) \, \le \, \Pb(e\in T) \.\cdot\. \Pb(e'\in T),
\end{equation}
where the probability is over uniform spanning trees~$T$ in a given
graph \ts $G=(V,E)$. This idea has proved extremely influential over the past
decades, leading to a long series of remarkable results.

In Graph Theory, the classic examples of correlation inequalities for hereditary
graph properties include the \defng{Harris inequality} (1960), the
\defng{Kleitman inequality} (1966) 
and their far-reaching
extensions: the \defng{Ahlswede--Daykin} (1978) 
and the \defng{FKG inequalities} (1971).  
We refer to \cite[$\S$6.1]{AS} for the introduction, and to \cite{FS} for a detailed overview.

In Statistical Physics, notable correlation
inequalities include \defng{Griffiths} (1967) 
and \defng{Griffiths--Hurst--Sherman inequalities} (1970), 
all for the Ising model.  We refer to \cite[Ch.~15]{AB} for the background 
and combinatorial applications.
See also \cite{BBL,KN,Pem} for the general theory of negative correlations. 

In Enumerative and Algebraic Combinatorics, correlation inequalities are closely related to
log-concavity (log-convexity) of sequences counting various combinatorial objects,
see~\cite{Bra,Sta2}.  
%
In the past few years, Huh and his coauthors introduced and explored a new
algebro-geometric approach to the subject.  We refer to \cite{AHK} for a major
breakthrough proving a long-standing Mason's conjecture on the numbers of independent
sets of a matroid, to \cite{HSW} for correlation inequalities for matroids,
and to \cite{Huh,Huh2} for recent surveys.
 We also mention the remarkable papers \cite{ALOV,BH} \ts
 which used \defng{Lorentzian polynomials} to obtain (stronger) ultra-log-concave
 inequalities for matroids.

In Poset Theory, besides Stanley inequality (Theorem~\ref{t:poset-Stanley}), notable
correlation inequalities include the \defng{Graham--Yao-Yao inequality} (1980),
 the \defng{XYZ inequality} (1982), and the \defng{Kahn--Saks inequality} \cite{KS},
 see an overview in \cite{Fish,Win}.
We also note the (conjectural) \defng{cross--product inequality}, see~$\S$\ref{ss:finrem-CPC}.

As we mentioned earlier, this paper grew from our paper \cite{CP1}
where we introduced the combinatorial atlas technology, using an involved
inductive argument based on linear algebra.  See also~\cite{CP2} for an
introduction to this theory, and for connections to Lorentzian polynomials and the
Alexandrov--Fenchel inequality (see also $\S$\ref{ss:finrem-Stanley}).

In contrast with the Lorentzian polynomial approach,  combinatorial
atlas is fundamentally noncommutative.  This and the elementary nature
of matrices makes the combinatorial atlas technology flexible enough to apply
to linear extensions and obtain the inequalities that are not reachable
by other means.  We believe that this applies to \cite[Thm~1.35]{CP1}
and most new inequalities for linear extension in this paper.
See~$\S$\ref{ss:finrem-easy} for further discussion on this.

\medskip

\subsection{Discussion}\label{ss:intro-disc}
Linear extensions of finite posets are somewhat different in spirit from
graphs, matroids, etc., in the type of statistics worth studying, and 
the difficulty of available tools, see \cite{CP23-survey}. 
Of the fairly large pool of potential correlation inequalities
we chose the ones that were accessible with the atlas technology.
The following two examples were motivational for our work.

\smallskip

\nin
{\small $(1)$} \.
It is instructive to compare our results with the \defn{Huh--Schr\"oter--Wang correlation
inequality} \ts for \defng{bases of matroids} \ts \cite{HSW}, which can be stated as follows.
Let \ts $\Mf=(X,\Bc)$ \ts be a finite matroid
of rank \ts $d=\rk(\Mf)$ \ts given by the set of bases \ts $\Bc \subseteq \binom{X}{d}$.
Then, for every two distinct elements \ts $u,v\in X$, we have:
\begin{equation}\label{eq:HSW}
\frac{\Pb[u \in B, v\in B] \.\cdot \.  \Pb[u \notin B, v\notin B] }{\Pb[u \in B, v\notin B]  \.\cdot \. \Pb[u \notin B, v\in B] }
\ \le  \ 2  \left(1 \. - \. \frac1d\right),
\end{equation}
where the probability is over uniform \ts $B\in \Bc$, see \cite[Thm~5]{HSW}.
The authors note that this implies the covariance bound
\begin{equation}\label{eq:HSW2}
\Covb(u\in B, v\in B) \, < \, \Pb[u \in B] \.\cdot \.\Pb[v\in B]\ts,
\end{equation}
for all fixed distinct elements \ts $u,v\in X$.
Although we are not aware of a formal connection, our upper bounds
\eqref{eq:poset-corr-tight-probab} and \eqref{eq:poset-cov-2} have a similar
in structure to \eqref{eq:HSW} and \eqref{eq:HSW2}, respectively.

Continuing the analogy, the constant~$2$ in~\eqref{eq:HSW} was improved to~$1$
in some special cases \cite[Thm~6]{HSW}. It is not known whether it can be further
improved in full generality, with some examples giving $8/7$ ratio.  Similarly,
our tools do not allow the asymptotic improvement of the constant~$2$.
in~\eqref{eq:poset-cov}, but we do have examples showing that here the constant~2 is asymptotically tight (Example~\ref{ex:poset-AK}).

\smallskip

\nin
{\small $(2)$} \.
In a related earlier study of \defng{random graph matchings}, Kahn \cite[$\S$4]{Kahn}
introduced the notion of \defna{approximate independence} \ts to describe
correlation inequalities of the type we consider in this paper.
%
Let $G=(V,E)$ be a simple graph. A \defn{matching} \ts is a subset \ts $M\subseteq E$ \ts
of pairwise nonadjacent edges.  Let \ts $\cM$ \ts denote the set of all matchings in~$G$,
and let \ts $\cM_k\subseteq\cM$ \ts be the set of matchings in~$G$ of size~$k$.

Heilmann and Lieb (1972) 
famously showed that the sequence \ts $\big\{|\cM_k|\big\}$ \ts is
log-concave, which parallels Stanley inequality \eqref{eq:Stanley}.
%
\defn{Kahn inequality} \cite[Cor.~4.3]{Kahn} states
that for all distinct vertices \ts $u,v\in V$, we have:\footnote{In fact, Kahn's bound is
stronger on both sides. We present a simplified version for clarity.}
\begin{equation}\label{eq:Kahn}
\frac12 \, \le \, \frac{\Pb[u\not\in M, \ts v\not\in M]}{\Pb[u\not\in M]\. \cdot \. \Pb[v\not\in M]} \, \le \, 2\ts,
\end{equation}
where \ts $M \in \cM$ \ts is a uniform random matching (cf.\ Example~\ref{ex:Kahn-zeta}).
Our correlation inequality \eqref{eq:poset-corr-tight-probab}
is modeled after this result.

\bigskip

\subsection{Paper structure}\label{ss:intro-structure}
We start with notations in Section~\ref{s:notation}.  In Section~\ref{s:app},
we present a number of examples and applications of correlation inequalities.
Many of these are for illustrative purposes and can be skipped, but some give
background and context for the type of correlation inequalities that we consider
throughout the paper.

We present applications to Enumerative Combinatorics in Section~\ref{s:enum},
notably to inequalities between numbers of certain \defng{standard Young tableaux},
and to polynomial inequalities for \defng{Euler numbers}.  Although these are
results not mentioned in the introduction to avoid shifting the focus,
the reader might find them curious (cf.~$\S$\ref{ss:finrem-Bjo}).  While we aim to be
self-contained and include most definitions, some familiarity with the
subject is helpful for the clarity and the motivation.

In Section~\ref{s:atlas}, we give a linear algebraic setup to derive
correlation inequalities for hyperbolic matrices. This section is completely
self-contained.  The hope is that our basic approach can be used to obtain
other inequalities.  A lengthy Section~\ref{s:proof} contains proofs
and occasional generalizations of results in the introduction.  The key
result there is Proposition~\ref{prop:M} proved in \cite{CP1} and used
as a black box in this paper.

A short Section~\ref{s:log-proof} has proofs of three implications of
the log-concavity.  All of these are straightforward and included for
completeness.  We conclude with final remarks and open problems in
Section~\ref{s:finrem}.

\bigskip

\section{Notations} \label{s:notation}

We use \ts $\nn=\{0,1,2,\ldots\}$, \ts $\rr_+ = \{x\ge 0\}$ \ts and \ts $[n]=\{1,2,\ldots,n\}$.
For a poset \ts $P=(X,\prec)$, denote by \ts $P^\ast = (X, \prec^\ast)$ \ts the
\defn{dual poset} \ts with \ts $x\prec^\ast y$ \ts if and only if \ts $y\prec x$,
for all \ts $x,y \in X$.  Denote by \ts $\min(P)$ \ts and \ts $\max(P)$ \ts the set
of minimal and maximal elements in~$P$, respectively.
For $x,y \in X$, we write \. $x \. || \. y $ \. if $x$ is incomparable to $y$  in $P$.

For posets \ts $P=(X,\prec_P)$ \ts and \ts $Q=(Y,\prec_Q)$,
the \defnb{parallel sum} \. $P+Q=(Z,\prec)$ \. is the poset on the disjoint union \ts $Z=X \sqcup Y$,
where elements of $X$ retain the partial order of~$P$, elements of $Y$ retain
the partial order of $Q$, and elements \ts $x \in X$ \ts and  \ts $y \in Y$ are incomparable.
Similarly, the \defnb{linear sum} \. $P\oplus Q = (Z,\prec)$, where \ts $x\prec y$ \ts for every
two elements \ts $x \in X$ \ts and  \ts $y \in Y$ \ts and other relations as in the parallel sum.
For an element \ts $a\in X$, denote by \. $a  \!\uparrow \. := \ts \{x\in X \. : \. x\succcurlyeq a\}$ \.
the \defn{upper order ideal} generated by~$a$.
For
a subset \ts $A\subseteq X$, denote by \ts $A\!\uparrow \. := \ts \cup_{a\in A} \, a\!\uparrow$ \. the
\defn{upper closure} of~$A$. We refer to \cite[Ch.~3]{EC} and \cite{Tro} for more on
poset definitions and notation.

We use capitalized bold letters to denote matrices, e.g.\ \ts $\bM=(\aM_{ij})$.  We also keep
conventional index notations, so, e.g., \. $\big( \bM+ \bM^2 \big)_{ij}$ \. is the $(i,j)$-th
entry of \. $\bM+\bM^2$.
Similarly, we use small bold letters to denote vectors, e.g.\ \ts
$\hb = (\ah_1,\ah_2,\ldots)$ \ts and \ts $\ah_i= (\hb)_i$.
For a subset \ts $S\subseteq [n]$,
the \defn{characteristic vector} \ts of~$S$  is the vector \. $\vb \in \Rb^n$ \. such that
\ts $\av_i=1$ \ts if \ts $i \in S$ and \ts $\av_i=0$ \ts if \ts $i \notin S$.
We denote by $\mathbf{0}$ the zero vector $(0,\ldots, 0)$.

%

\bigskip



\section{Examples and special cases}\label{s:app}

\smallskip

\subsection{Basic examples}\label{ss:app-simple}
Here we consider several simple observations on correlation inequalities that
apply to general posets.  We recommend reading these first, before
studying more elaborate examples with special families of posets.

\smallskip

\begin{ex} \label{ex:poset-corr-upper-two}
Suppose \ts $\min(P)=\{x,y\}$, i.e.\ \ts $x, y$  \ts are
 the only minimal elements of~$P$. We then have
\ts $e(P)  = e(P-x) + e(P-y)$.  The upper bound
of~\eqref{eq:poset-corr-tight} is straightforward in this case:
\begin{align*}
	 \frac{e(P) \.\cdot \. e(P-x-y)}{e(P-x) \.\cdot \.  e(P-y)}  \ = \
\frac{e(P-x-y)}{e(P-x)} \, + \, \frac{e(P-x-y)}{e(P-y)} \  \leq \ 2\ts,
\end{align*}
where the inequality follows from  \. $e(P-x-y) \ts \le \ts \min\big\{e(P-x), \ts e(P-y)\big\}$.
See also~$\S$\ref{ss:finrem-easy} for a further discussion of this proof.
\end{ex}

\smallskip

\begin{ex} \label{ex:poset-corr-upper-two-more}
Again, suppose \ts $\min(P)=\{x,y\}$.  In the notation of $\S$\ref{ss:intro-Stanley},  we have: \.
$$\Pb\big[ f(x)  =  1\big] \. = \. \frac{e(P-x)}{e(P)}\., \ \quad
\Pb\big[ f(x) >  1\big] \. = \.  \frac{e(P-y)}{e(P)} \ts \quad \text{and} \ts \quad
\Pb\big[f(x)  =   2\big]  \. = \. \frac{e(P-x-y)}{e(P)}\ts.
$$
Then \eqref{eq:Stanley-cor2}  becomes \.
$$
\frac{e(P-x-y)  \.\cdot \.  e(P)}{e(P-x)  \.\cdot \.  e(P-y)} \, \ge \, 1\ts.
$$
This inequality follows from and is asymptotically equivalent to the lower bound
in~\eqref{eq:poset-corr-tight}.
\end{ex}

\smallskip

\begin{ex} \label{ex:poset-narrow-closure}
For a subset \ts $A\subset \min(P)$ \ts of minimal elements,
define \. $A^\bd$ \. as follows:
$$
 A^\bd \, := \, \big\{\ts x\in X \ : \ x \succcurlyeq y, \ y \in \min(P) \ \Rightarrow \ y \in A \ts \big\}.
$$
Equivalently, \. $A^\bd  =  A\!\uparrow  \sm \.  B\!\uparrow$, where \ts $B:=\min(P)-A$.
Clearly, we have \ts $A^\bd \subseteq A\!\uparrow$, and sometimes the subset \ts $A^\bd$ \ts
is much smaller than \ts $A\!\uparrow$.  On the other hand, we have:
$$
\Pb\big[1 \in f(A), \ts 2\in f(A\!\uparrow)\big] \, = \, \Pb\big[1 \in f(A), \ts 2\in f(A^\bd)\big].
$$
This is because the only way we can have \ts $f(y)=2$ \ts and \ts $f(x)=1$ \ts for some \ts $x\in A$,
is if either \ts $y\in \min(P)$, or \ts $y$ \ts covers element~$x$ and no other element.
This argument is useful in trying understand
our correlation inequalities, as the next example shows.
\end{ex}

\smallskip

\begin{ex} \label{ex:poset-upper-closure}
Let \ts $A\subset \min(P)$ \ts and let \ts $B := \min(P)-A$\ts.
By applying \eqref{eq:Stanley-ext3}
to $A\!\uparrow$,  we get:
\begin{equation}\label{eq:poset-upper-ex1}
\Pb[1\in f(A), 2\notin f(A\!\uparrow)]^2 \ \geq \ \Pb[1 \in f(A), \ts 2\in f(A\!\uparrow)] \. \cdot \. \Pb[1 \in f(B), \ts 2\notin f(A\!\uparrow)]\..
\end{equation}
Alternatively,  apply \eqref{eq:poset-disjoint-logconcave-new} with this choice of $A,B$. We get:
\begin{equation}\label{eq:poset-upper-ex2}
\Pb\big[1\in f(A), 2 \in f(B)\big]^2 \ \geq \ \Pb\big[1\in f(A), \ts 2 \in f(A\!\uparrow)\ts\big] \. \cdot \. \Pb\big[1\in f(B), \ts 2 \in f(B\!\uparrow)\big]\ts. \quad \ \
\end{equation}
At first sight, the LHS of \eqref{eq:poset-upper-ex1} $\ge$ LHS of \eqref{eq:poset-upper-ex2}, while
the RHS of \eqref{eq:poset-upper-ex1} $\le$ RHS of \eqref{eq:poset-upper-ex2}. The argument in
the previous example shows that both of these are equalities.  In other words,
the inequality \eqref{eq:Stanley-ext3} is equivalent to the inequality
\eqref{eq:poset-disjoint-logconcave-new} in this case.
\end{ex}

\medskip

\subsection{Second moment bound}\label{ss:app-second}
%
%
In Theorem~\ref{t:poset-covariances}, letting \ts $y= x$ \ts
gives the following curious bound on the second moment:

\smallskip

\begin{cor}\label{c:poset-variances}
Let \ts $P=(X,\prec)$ \ts be a finite poset, and let \ts $x \in X$ \ts be a fixed element.
Then:
\begin{equation}\label{eq:poset-exp}
1 \, \le \	\frac{\Eb[f(x)^2 ]}{\Eb[f(x)]^2} \ < \, 2 \ts.
\end{equation}
\end{cor}

\smallskip

Here the lower bound is trivial:  \ts $0\le \Varb(Z) = \Eb[Z^2] -\Eb[Z]^2$ \ts for every random variable~$Z$.
Note also that the corollary immediately implies a weaker version of \eqref{eq:poset-cov}:
\begin{equation}\label{eq:poset-cov-more}
\frac{\Eb[f(x) \ts f(y)]}{\Eb[f(x)] \.\cdot\. \Eb[f(y)]} \ < \, 2 \ts.
\end{equation}
This follows by combining the upper bound in \eqref{eq:poset-exp} with \. $\Cov(X,Y)\le \Varb(X) \cdot \Varb(Y)$.

\smallskip

Note that our proof of Theorem~\ref{t:poset-covariances} follows from a construction similar
to that used in~\cite{CP1} to rederive Stanley Theorem~\ref{t:poset-Stanley}.
Since Corollary~\ref{c:poset-variances} is a simple consequence of
Theorem~\ref{t:poset-covariances}, let us show that it follows  from
Stanley theorem as well.

\smallskip

\begin{prop}\label{p:log-Petrov}
Let \ts $Z$ \ts be a random variable on \ts $\{1,2,\ldots\}$ \ts with a
log-concave distribution.  Then \. $\Eb[Z^2] \ts \le \ts 2\ts \Eb[Z]^2$.
\end{prop}

\smallskip

Note that constant~$2$ in the proposition is tight, as shown by a geometric
random variable \ts $Z=$ Geo$(q)$, where we let \ts $q\to 1$.   We prove
Proposition~\ref{p:log-Petrov} in~$\S$\ref{ss:cor-proof-log-lemma}.
On the other hand, the constant~$2$ in \eqref{eq:poset-exp}
can be lowered in many cases.
\smallskip

\begin{ex}\label{ex:poset-43}
	Let \ts $P:=C_{n-1}+\{x\}$ \ts be the parallel sum of two chains of size $n-1$ and $1$.  We have:
	\[\frac{\Eb[f(x)^2 ]}{\Eb[f(x)]^2}
	\ = \ \frac{\frac{1}{n} \. \sum_{k=1}^n \. k^2}{\big(\frac{1}{n} \. \sum_{k=1}^n \. k \big)^2} \ = \    \frac{4n+2}{3n+3} \
	\to \ \frac{4}{3} \ \ \quad \text{as \ \, $n \to \infty$\ts.}
	\]
\end{ex}

\smallskip

This example motivated us to conjecture that the bound \ts $\frac43$ \ts is sharp.   
The following counterexample was given by Max Aires and Jeff Kahn.\footnote{July 2024, personal communication}

\smallskip

\begin{ex}\label{ex:poset-AK}
Let \ts $P:=C_{m}+C_{n-m}$ \ts be the parallel sum of two chains of size $m$ and $n-m$, 
such that \ts $m=o(n)$\ts.
Let $x$ be the minimum element in the chain $C_m$. We have:
\[
\Eb[f(x)] \, = \, \frac{1}{\binom{n}{m}} \. \sum_{k=1}^{n-m+1} \. k \binom{n-k}{m-1}
 \qquad \text{and}
\qquad  
\Eb[f(x)^2] \, = \, \frac{1}{\binom{n}{m}} \. \sum_{k=1}^{n-m+1} \. k^2 \binom{n-k}{m-1}.
\]
This gives
\[
\frac{\Eb[f(x)^2 ]}{\Eb[f(x)]^2}  \ \to \ 2 \ \ \quad \text{as \ \, $m, \. n \. \to \. \infty$\ts.}
	\]
In particular, this shows that the upper bound in \eqref{eq:poset-exp} and Theorem~\ref{t:poset-covariances} 
is asymptotically tight.
\end{ex}

\smallskip

\begin{ex}\label{ex:Kahn-zeta}
Proposition~\ref{p:log-Petrov} is sharp in full generality, but can be weak in other
instances. In the notation of $\S$\ref{ss:intro-disc}, let~$G$ be simple graph on $2n$ vertices.
Denote by \ts $\al_G(M)$ \ts be the size of a random matching~$M$,
and let \ts $\ze_G(M):=n-\al_G$ \ts be half of the number of vertices not covered by~$M$.
The Heilmann--Lieb theorem proves that \ts $\ze_G$ \ts has a log-concave distribution.
Proposition~\ref{p:log-Petrov} then implies \.
$\Varb(\ze_G) = \Eb\big[\ze_G^2\big] - \Eb[\ze_G]^2 <  \ts \Eb[\ze_G]^2$.
On the other hand, \cite[Cor.~4.2]{Kahn} gives a much
stronger bound \. $\Varb(\ze_G) \le \Eb[\ze_G]$, which can then be used to obtain
concentration inequalities (see e.g.\ \cite[Ch.~7]{AS}).
\end{ex}

\smallskip

\subsection{Posets with unique covers of minimal elements}\label{ss:app-general}
%
%
%
%
For elements \ts $x,y\in X$ \ts in a poset $P=(X,\prec)$,
we say that element $y$ \defn{covers} \ts $x$, if \ts $x\prec y$,
and there is no \ts $v\in X$ \ts s.t.\ $x\prec v \prec y$.
For elements \ts $x\prec y$ \ts in~$X$, we say that $y$ is a
\defn{unique cover} \ts of~$x$, if
$$(\lozenge) \quad \text{$y$ \ts covers \ts $x$ \ts and does not cover any other elements in~$X$.}
$$
The following correlation inequality is surprising even
in the most simple special cases (see below).

\smallskip

\begin{cor}\label{c:poset-disjoint-cor}
Let \ts $P=(X,\prec)$ \ts be a finite poset, and let \ts $x,y\in \min(P)$ \ts be distinct
minimal elements.  Suppose element \ts $v\in X$ \ts is a unique cover of~$x$, and \ts $w \in X$ \ts
is a  unique cover of~$y$. Then:
\begin{equation}\label{eq:poset-disjoint-logconcave-var}
	  e(P-x-y)^2 \ \geq \ e(P-x-v) \.\cdot \. e(P-y-w).
\end{equation}
\end{cor}

\begin{proof}
In the notation of Theorem~\ref{t:poset-disjoint-logconcave},
let \ts $A=\{x\}$ \ts and \ts $B=\{y\}$.  Note that
$$
\Pb\big[1\in f(A), \ts 2 \in f(A\!\uparrow)\ts\big] \, = \, \Pb\big[f(x)=1, \ts f^{-1}(2) \succ x\ts\big]
\, \geq  \, \Pb\big[f(x)=1, \ts f(v)=2\ts\big] \, = \, \frac{e(P-x-v)}{e(P)}\..
$$
Here the inequality follows from the definition of~$v$ as an  element which satisfies~$(\lozenge)$.
By the same argument we have:
$$
\Pb\big[1\in f(B), \ts 2 \in f(B\!\uparrow)\big] \, \geq  \, \frac{e(P-y-w)}{e(P)}\..
$$
On the other hand, we also have:
$$
\Pb\big[1\in f(A), \ts 2 \in f(B)\big] \, = \, \frac{e(P-x-y)}{e(P)}\..
$$
Now \eqref{eq:poset-disjoint-logconcave-new} implies the result.
\end{proof}

\smallskip

We conclude this section with another four-element inequality:

\smallskip

\begin{cor}\label{c:poset-disjoint-three}
Let \ts $P=(X,\prec)$ \ts be a finite poset, and let \ts $x,y,z\in \min(P)$ \ts
be distinct minimal elements.   Suppose element \ts $u\in X$ \ts is a unique
cover of~$z$.  Then:
\begin{equation}\label{eq:poset-disjoint-three}
	  e(P-u-z) \. e(P-x-y) \ \leq \ 2 \. e(P-x-z) \. e(P-y-z).
\end{equation}
\end{cor}

\begin{proof}
Let \ts $A=\{x\}$, \ts $B=\{y\}$ \ts and \ts $C=\{z\}$.  The corollary now follows
from Theorem~\ref{t:poset-disjoint-half-correlation} and the argument in the proof
of the Corollary~\ref{c:poset-disjoint-cor}.
\end{proof}

\bigskip

\section{Enumerative applications}\label{s:enum}


\subsection{Standard Young tableaux} \label{ss:app-def}
Let \ts $\la=(\la_1,\ldots,\la_\ell)$ \ts be an integer partition of~$n$,
write \ts $\la \vdash n$.  Here \ts $\ell(\la)$ \ts denotes the number of parts.
A Young diagram is a set \. $Y_\la:=\big\{(i,j)\in \nn^2 \. : \. 1\le j \leq  \la_i,
\ts 1\le i \le \ell\big\}$. A \defn{standard Young tableau} \ts of shape $\la$ \ts
is a bijection \ts $f: Y_\la\to [n]$ \ts which increases in both directions.
Denote by \ts $\SYT(\la)$ \ts the set of standard Young tableaux of shape~$\la$.
The number \ts $|\SYT(\la)|$ \ts can be computed by the
\defng{hook-length formula}, see e.g.\ \cite[$\S$7.21]{EC}.

A \defn{conjugate partition} \ts $\la'=(\la_1',\la_2',\ldots)$ \ts is defined by
\ts $\la_i' = |\{i\.: \la_i\ge i\}$.  Note that Young diagrams \ts $Y_\la$ \ts
and \ts $Y_{\la'}$ \ts are dual with respect to the \ts $i=j$ \ts reflection.
Partition \ts $\la$ \ts is called \defn{self-conjugate} \ts  if \ts $\la=\la'$.
For partitions \ts $\la$ \ts and \ts $\mu$, define the \defn{sum} \.
$\la+\mu := (\la_1+\mu_1, \la_2+\mu_2, \ldots)$ \ts and the \defn{union} \.
$\la\cup \mu := (\la'+\mu')'$.

Two squares \ts $x\prec y$ \ts are called \defn{adjacent} \ts if \ts $x=(i,j)$ \ts and \ts
$y=(i,j+1)$ \ts or \ts $y=(i+1,j)$.
The \defn{corners} \ts of~$\la$ are defined as elements \. $(i,j)\in Y_\la$ \. such that
\. $(i+1,j), (i,j+1)\notin Y_\la$\ts.
Denote by \ts $\Cc(\la)$ \ts the set of corners of~$\la$.  Similarly, the
\defn{boundary squares} \ts of~$\la$ are defined as elements \. $(i,j)\in Y_\la$ \. such that either
\ts $(i+1,j)$ \ts or \ts $(i,j+1)$ \ts is not in \ts $Y_\la$\ts.  Denote by \ts $\Dc(\la)$ \ts
the set of boundary squares of~$\la$, and note that \. $\Cc(\la) \subseteq \Dc(\la)$.

Let \ts $\mu=(\mu_1,\mu_2,\ldots)$ \ts be a partition such that \.
$\mu_i \le \la_i$ \. for all \ts $0\le i \le \ell$.  The set \.
$Y_{\la/\mu} := Y_\la \sm Y_\mu$ \. is the
\defng{skew Young diagram} \ts  of (skew) shape \ts $\la/\mu$.  Standard Young
tableaux of skew shapes are defined similarly to the usual (straight) shapes.
The number \ts $|\SYT(\la/\mu)|$ \ts  of standard Young tableaux
of shape $\la/\mu$ \ts can be computed by the
\defng{Aitken--Feit determinant formula}, see e.g.~\cite[$\S$7.16]{EC}.
The \defn{corners} \ts of~$\la/\mu$ are defined as elements \. $(i,j)\in Y_{\la/\mu}$ \. such that
\. $(i+1,j), (i,j+1)\notin Y_{\la/\mu}$.

\medskip

\subsection{Removing the corners}\label{ss:app-Young}
We now apply the correlation inequalities to posets corresponding to skew Young diagrams.
As we do, we get both new and familiar inequalities (see below).

\smallskip

\begin{cor} \label{c:YT1}
Let $\la/\mu$ be a skew shape, let \ts $x,y\in \Cc(\la/\mu)$ \ts be
corners, and let \ts $v,w\in \Dc(\la)$ \ts be a boundary square adjacent
to~$x$ and~$y$, respectively. Then:
\begin{equation}\label{eq:Stanley-SYT}
|\SYT(\la/\mu-x-y)|^2 \ \ge \ |\SYT(\la/\mu-x-v)| \.\cdot \. |\SYT(\la/\mu-y-w)|\ts.
\end{equation}
In particular, if \ts $\la$ \ts and \ts $\mu$ \ts are self-conjugate,
\ts $x=(i,j)$ \ts and \ts $y=(j,i)$, then:
\begin{equation}\label{eq:Stanley-SYT-sc}
|\SYT(\la/\mu-x-y)| \, \ge \, |\SYT(\la/\mu-x-v)|\ts.
\end{equation}
\end{cor}


\begin{figure}[hbt]
\begin{center}
	\includegraphics[height=3.2cm]{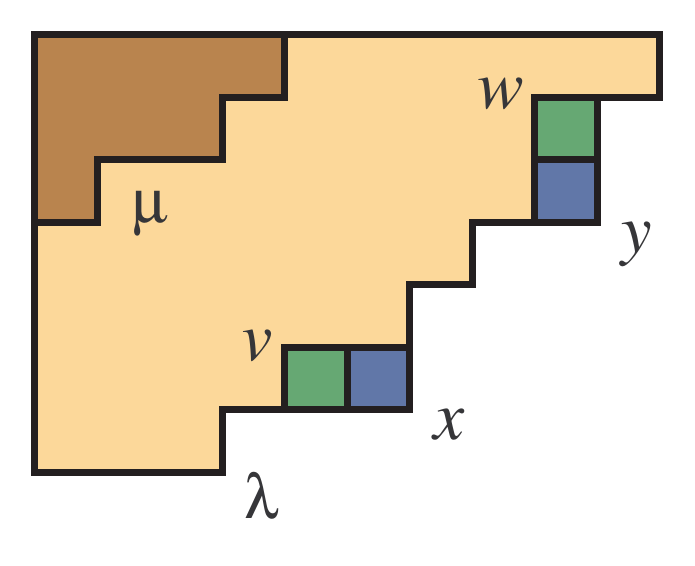}
\vskip-.3cm
\end{center}
\caption{Skew shape \ts $\la/\mu$, where \ts $\la=(10,9,9,7,6,6,3)$, \ts $\mu=(4,3,1)$,
	corners \ts $x,y\in \Cc(\la)$ and boundary squares \ts $v,w\in \Dc(\la)$.}
\label{f:YT}
\end{figure}


\begin{proof}
Let poset \ts $\cP_{\la/\mu}=(Y_{\la/\mu},\prec)$ \ts
be defined by \. $(i,j) \preccurlyeq (i',j')$ \. if \ts $i\le i'$ \ts
and \ts $j\le j'$.  
The set of linear extensions \ts $\Ec(\cP_{\la/\mu})$ \ts is in bijection with $\SYT(\la/\mu)$,
so \ts $e(\cP_{\la/\mu}) = |\SYT(\la/\mu)|$.
Note that  minimal elements in the dual poset \ts $\cP_\la^\ast$ \ts are the corners of~$\la$: \.
$\min(\cP_\la^\ast) = \Cc(\la)$. Similarly, for skew shapes
we have: \. $\min(\cP_{\la/\mu}^\ast) = \Cc(\la)\sm \Cc(\mu)$.

In notation of~$\S$\ref{ss:app-general}, note that \ts $v$ \ts and \ts $w$ \ts
are unique covers of \ts $x$ \ts and \ts $y$, respectively.  Now Corollary~\ref{c:poset-disjoint-cor}
gives \eqref{eq:Stanley-SYT}.
For the second part, let \ts $v=(p,q)$ \ts and take \ts $w:=(q,p)$.  Now~\eqref{eq:Stanley-SYT}
implies \eqref{eq:Stanley-SYT-sc}.
\end{proof}

\smallskip

\begin{cor} \label{c:YT-log}
Let \ts $\la=(\la_1,\ldots,\la_\ell)$ \ts be a partition.
We have the following log-concave inequality:
\begin{equation}\label{eq:Stanley-SYT-three}
\big|\SYT\big(\la/(a,1^b)\big)\big|^2  \ \ge \  \big|\SYT\big(\la/(a+1,1^{b-1})\big)\big| \.\cdot \.
\big|\SYT\big(\la/(a-1,1^{b+1})\big)\big|\ts,
\end{equation}
for all \ts $1< a < \la_1$ \ts and \ts $1< b < \ell$.
\end{cor}


\begin{proof}
Following the argument in the proof above, let \ts $\cP_{\la/\mu}=(Y_{\la/\mu},\prec)$ \ts
where \ts $\mu = (a+1,1^{b+1})$, and let \ts $x=(1,a)$, \ts $y=(b,1)$, \ts $v=(1,a+1)$, \ts
$w=(b+1,1)$.   Now Corollary~\ref{c:poset-disjoint-cor} gives the result.
\end{proof}
\smallskip

\begin{cor} \label{c:YT2}
Let $\la/\mu$ be a skew shape, let \ts $x,y,z\in \Cc(\la/\mu)$ \ts be
corners, and let \ts $u\in \Dc(\la)$ \ts be a boundary square adjacent
to~$z$, respectively. Then:
\begin{equation}\label{eq:Stanley-SYT2-three}
\frac{|\SYT(\la/\mu-u-z)| \. \cdot \. |\SYT(\la/\mu-x-y)|}{
|\SYT(\la/\mu-x-z)| \.\cdot \. |\SYT(\la/\mu-y-z)|}  \ \le \ 2\ts.
\end{equation}
\end{cor}


The proof follows from Corollary~\ref{c:poset-disjoint-three} along the same lines
as the proofs above.  We omit the details.

\medskip

\subsection{Related inequalities}\label{ss:SYT-related}
Corollary~\ref{c:YT1} is closely related to two other inequalities.
The first is the \defn{Okounkov inequality} \ts conjectured
in \cite[p.~269]{Oko}, and proved by
Lam--Postnikov--Pylyavskyy \cite[Thm~4]{LPP07}.  In a special case,  
the Okounkov inequality gives:  
\begin{equation}\label{eq:Oko}
\big|\SYT\big(\tfrac{\la+\nu}2 \ts \big/ \ts \tfrac{\mu+\al}2 \big)\big|^2 \ \ge \ |\SYT(\la/\mu)| \.\cdot \. |\SYT(\nu/\al)|\ts,
\end{equation}
for all skew shapes \ts $\la/\mu$ \ts and \ts $\nu/\al$, such that  partitions \ts $\la+\nu$ \ts and \ts $\mu+\al$ \ts
have even parts, cf.~$\S$\ref{ss:finrem-Bjo} and~\cite{CP3}. 

In the notation of Corollary~\ref{c:YT1}, observe that for $v$ to the left of~$x$,
and for $w$ to the left of~$y$ (see Figure~\ref{f:YT}), the Okounkov inequality
coincides with our inequality \eqref{eq:Stanley-SYT}.  On the other hand,
even the inequality \eqref{eq:Stanley-SYT-sc} does not follow from this argument.

\smallskip

The second is the \defn{Fomin--Fulton--Li--Poon {\rm $($FFLP$)$} inequality} \ts conjectured in \cite[Conj.~2.7]{FFLP}, and  also proved in \cite[Thm~4]{LPP07}.  Again, in a special case,
this inequality gives:  
\begin{equation}\label{eq:FFLP}
\aligned
& {\big|\SYT\big( \sort_1(\la,\nu) \ts / \ts \sort_1(\mu,\al)\big)\big| \. \cdot \. \big|\SYT\big( \sort_2(\la,\nu) \ts / \ts \sort_2(\mu,\al)\big)\big|} \\
&  \hskip2.9cm \ge \ {|\SYT(\la/\mu)| \.\cdot \. |\SYT(\nu/\al)|\ts.}
\endaligned
\end{equation}
Here for two partitions \ts $\be$ \ts and $\ga$ \ts we define two other partitions \. $\sort_1(\be,\ga):= (\tau_1,\tau_3,\ldots)$ \. and
\. $\sort_2(\la,\mu):= (\tau_2,\tau_4,\ldots)$, where \. $(\tau_1,\tau_2,\ldots):=\be\cup \ga$.

In notation of Corollary~\ref{c:YT1}, observe that for $v$ above~$x$,
and for $w$ to the above~$y$, the FFLP inequality coincides with our
inequality \eqref{eq:Stanley-SYT}.  In this sense it is a conjugate dual
to the Okounkov inequality \eqref{eq:Oko}, cf.~\cite{LPP07}.  On the other
hand, for $v$ to the left of~$x$,
and for $w$  above~$y$, the LHS of \eqref{eq:FFLP} has partitions of unequal size:
\begin{equation}\label{eq:FFLP-corners} \quad
|\SYT(\la/\mu-y)| \. \cdot \. |\SYT(\la/\mu-x-y-v)| \ \ge \ |\SYT(\la/\mu-x-v)| \.\cdot \. |\SYT(\la/\mu-y-w)|\ts.
\end{equation}
In summary, the inequality \eqref{eq:FFLP-corners} follows from known inequalities 
in some cases (where it extends to inequalities for symmetric functions), and is 
new in other cases (where it does not).  

\medskip

\subsection{Hook walk}\label{ss:SYT-hook-walk}
Let \ts $f\in \SYT(\la)$ \ts be a random standard Young tableau
of shape \ts $\la\vdash n$. For a corner \ts $x\in \Cc(\la)$, denote by
$$
p(x) \, := \, \frac{|\SYT(\la-x)|}{|\SYT(\la)|}
$$
the probability distribution on \ts $\Cc(\la)$ \ts defined by the location
of~$n$ in~$f$.  The celebrated \defn{hook walk} \ts \cite{GNW79} shows how to
sample from \ts $p(x)$:

\smallskip

\qquad $\circ$ \ Start at a random square $y\in Y_\la$.

\qquad $\circ$ \ Move to a random square in the hook \ts $H(y)$.

\qquad $\circ$ \ Repeat until the walk reaches a corner $x\in \Cc(\la)$.

\smallskip

\noindent
Here \defn{the hook} \ts is defined as \.
$H(i,j) := \{(i,r) \. : \. j < r \le \la_i\} \cup \{(s,j) \. : \. i < s \le \la_j'\}$.
Let \ts $y\in \Cc(\la)$ \ts be a different corner.
One can think of \eqref{eq:poset-corr-tight-probab} as follows:
\begin{equation}\label{eq:HW}
  \frac{n}{n-1} \ \leq \ \frac{\Pb[f(y)=n-1 \.| \. f(x)=n]}{\Pb[f(y)=n]}  \ \leq \ 2.
\end{equation}
The lower bound in~\eqref{eq:HW} is now straightforward. Indeed, all hook walk
trajectories which arrive to $y$ in $(\la-x)$, also arrive in $\la$, and the
ratio is given by the ratio of sizes of diagrams in the starting location
of the walk.  The upper bound in \eqref{eq:HW} does not seem to follow 
directly from this probabilistic model.

%



\medskip

\subsection{Bruhat order} \label{ss:app-Bruhat}
Let \ts $\si\in S_n$ \ts and define the \defn{permutation poset} \ts $P_\si = ([n],\prec)$
\ts by letting
$$i \precc j \quad \Leftrightarrow \quad i\le j \ \ \text{and} \ \ \si(i)\le \si(j) \ts.
$$
It is easy to see that \ts $\Ec(P_\si)\subseteq S_n$ \ts is the lower
ideal \ts $[\id, \sigma]$ \ts of  \ts $\si$ \ts in the \defn{weak Bruhat order} \ts $\cB_n=(S_n,\vt)$,
see \cite{BW91,FW97}.  The set of minimal elements in \ts $\cP_\si$ \ts is the set of
\defn{minimal records} in~$\si$: \.
$$
\min(P_\si) \. = \. \big\{k\in [n] \, : \, \si(i)>\si(k) \ \ \text{for all} \ \ i<k\big\}.
$$

Fix two record minima \. $a,b \in \min(P_\si)$\ts.  We can then rewrite the correlation
inequality \eqref{eq:poset-corr-tight-probab} in terms of random permutations $\om \in S_n\ts$:
$$
\frac{n}{n-1} \ \le \ \frac{\Pb[\ts \om(a)=1, \ts \om(b)=2\. | \. \om \vte \si\ts]}{
\Pb[\ts \om(a)=1 \. | \.  \om \vte \si\ts] \. \cdot \.
\Pb[\ts \om(a)=2 \. | \.  \om \vte \si\ts]} \ \le \ 2\ts.
$$
As in the introduction, the lower bound is an equality for \ts $\si=(n,n-1,\ldots,1)$ \ts
and the upper bound is an equality for \ts $\si=(2,1,3,\ldots,n)$.


\medskip

\subsection{Zigzag posets} \label{ss:app-zigzag}
Let \ts $P_n:=(X,\prec)$ \ts be the poset on $n$ elements \ts
$X=\{x_1,y_1,x_2,y_2,x_3,\ldots\}$, with the only relations given by
\. $x_1\prec y_1 \succ x_2 \prec y_3 \succ x_3 \prec \ts \ldots$ \ts
Note that $P_n$ has height two and the set \ts $\Ec(P_n)$ \ts is in
bijection with \defn{alternating permutations} \ts
$\si(1) < \si(2) > \si(3) < \si(4) > \si(5) < \ldots$ \ts and the
standard Young tableaux of a \defng{zigzag shape} \cite{Sta3}.

It is well-known that \ts $e(P_n)=E_n$ \ts are the \defng{Euler numbers}, which satisfy
$$
\sum_{m=0}^{\infty} \. E_n \. \frac{t^{n}}{n!} \ = \ \tan(t) \. + \. \sec(t)\ts \quad
\text{and} \quad E_n \. \sim \, \frac{2^{n+2}\ts n!}{\pi^{n+1}} \ \ \. \text{as} \ \ n\to \infty,
$$
see e.g.\ \ts \cite[\href{http://oeis.org/A000111}{A000111}]{OEIS}, \ts
\cite[$\S$IV.6.1]{FS09} \ts and \ts \cite[$\S$1.6]{EC}.
We also have:
$$
\Pb[f(x_1)=k] \. = \. \frac{E_{n,k}}{E_k}\., \quad
\text{where} \quad
E_{n,k} \ := \ \sum_{i=0}^{\lfloor (k-1)/2\rfloor} \. (-1)^i \.  \binom{k}{2i+1} \. E_{n-2i-1}
$$
are the \defng{Entringer numbers}, see
\cite[\href{http://oeis.org/A008282}{A008282}]{OEIS}
and \cite{Sta3}.  Stanley Theorem~\ref{t:poset-Stanley} in this case gives
log-concavity of the Entringer numbers:
\begin{equation}\label{eq:Stanley-Entringer}
E_{n,k}^2 \, \ge \, E_{n,k-1} \.\cdot \. E_{n,k+1}\..
\end{equation}
This inequality can be viewed as a degree two polynomial inequality for the Euler numbers.
We refer to \cite{B+,G+} and \cite[$\S$1.38]{CP1} for more on \eqref{eq:Stanley-Entringer} and
its various generalizations.

\begin{ex}
Let \ts $n=2m+1$, and note that \ts $\min(P_n)=\{x_1,\ldots,x_{m+1}\}$.  Fix
 \ts $1\le k \le m+1$ \ts and consider a subset \ts $A:=\{x_1,\ldots,x_k\}$.
We have:
$$\Pb\big[1\in f(A)\big] \, = \, \frac{F_n(k)}{E_n}\., \quad
\Pb\big[1\notin f(A)\big] \, = \, \frac{F_n(m-k+1)}{E_n} \.,
\quad 
\Pb\big[1\notin f(A), 2\in f(A)\big] \, = \, \frac{G_n(k)}{E_n}\.,
$$
where
$$
\aligned
	F_n(k) \ & := \  \sum_{i=1}^{k} \. \binom{n-1}{2i-2} \. E_{2i-2} \. E_{n-2i+1} \qquad \text{and}  \\
    G_n(k) \ & := \   \sum_{i=1}^{k} \. \sum_{j=k+1}^{m+1} \.
    \binom{n-2}{2i-2,2j-2i-1} \.  E_{2i-2} \.   E_{2j-2i-1} \. E_{n-2j+1}\..
\endaligned
$$
Substituting these into \eqref{eq:Stanley-ext2} gives a new degree four
polynomial inequality for the Euler numbers:
\begin{equation}\label{eq:Euler-numbers-1}
F_n(k) \. \cdot \. F_n(m-k+1) \, \le \, E_n \. \cdot \. G_n(k).
\end{equation}
Note that for \ts $k, n-k=\om(1)$ \ts both sides have the same leading asymptotics, which makes
the inequality even more interesting.
\end{ex}

\begin{ex}
Let \ts $n=2m+1$ \ts and \ts $1\le k \le \ell \le m+1$.  Take
\ts $A=\{x_1,\ldots,x_k\}$ \ts as above.
We have:
$$\Pb\big[1,2\in f(A)\big] \, = \, \frac{H_n(k)}{E_n} \qquad \text{and} \qquad
\Pb\big[1,2\notin f(A)\big] \, = \, \frac{H_n(m-k+1)}{E_n}\.,
$$
where
$$
H_n(k) \ := \ 2 \. \sum_{1 \leq i <j \leq k} \.
\binom{n-2}{2i-2,2j-2i-1} \,  E_{2i-2} \.   E_{2j-2i-1} \. E_{n-2j+1}\..
$$
Substituting this and the inequalities from the previous example
into \eqref{eq:Stanley-ext3} with \. $A \gets X\sm A$ \.  gives a new degree six
polynomial inequality
for the Euler numbers:
\begin{equation}\label{eq:Euler-numbers-2}
G_n(k)^2 \ \geq \ H_n(k) \. \cdot \. H_n(m-k+1).
\end{equation}
\end{ex}

\bigskip



\section{Correlation inequalities from hyperbolic inequalities}\label{s:atlas}

In this short section we derive several correlation inequalities for hyperbolic matrices.
Although we are motivated by applications to combinatorial atlas \cite{CP2,CP1},
the presentation here is self-contained and uses nothing but linear algebra.

\smallskip

\subsection{Hyperbolic matrices}
Let \. $\bM = (\aM_{ij})$ \. be a symmetric \ts $d \times d$ \ts matrix with
entries \ts $\aM_{ij}\in \rr_+$.
The matrix  \. $\bM$ \. is \defn{hyperbolic} (satisfies \defn{hyperbolic property}), if
\begin{equation}\label{eq:Hyp}\tag{Hyp}
	\langle\xb, \bM \yb  \rangle^2 \ \geq \ \langle\xb, \bM \xb  \rangle \, \langle \yb, \bM \yb \rangle \quad \text{ for all  \ \ $\xb, \yb \in \Rb^{d}$ \ such that } \ \langle \yb, \bM \yb \rangle \geq 0.
\end{equation}
Note that \eqref{eq:Hyp} is equivalent to the matrix $\bM$ having at most one positive eigenvalue,
counting multiplicity (see e.g.\ \cite{CP2,SvH19}).

For every \. $\xb, \yb \in \Rb^d$\., we employ the following shorthand
\[  \bM_{\xb\!\yb} \  :=  \ \langle \xb, \bM \yb \rangle.  \]
Note that \. $\bM_{\xb \! \yb} = \bM_{\yb\!\xb}$ \. since \ts $\bM$ \ts is symmetric.

\smallskip

\begin{lemma}\label{lem:Pos}
	Let \. $\bMr$ \. be a nonnegative symmetric $d \times d$ matrix that satisfies \eqref{eq:Hyp},
and let \. $\xb,\yb,\zb \in \Rb^d_+$ \.  be nonnegative vectors.  Then either \.
$\bMr \. {\xb}= \bMr \. {\yb}=  \bMr \.{\zb} =\mathbf{0}$, or there exists \. $\vb \in \Rb^d_+$
such that \. $\langle \xb+\ep \vb, \bMr \. (\xb+\ep \vb) \rangle >0$ \. for all \ts $\ep >0$.
%
\end{lemma}

\smallskip

\begin{proof}
	We split the proof into two cases.
	First, suppose that \ts $\bM$ \ts has a positive eigenvalue \ts $\lambda>0$ \ts
    with a corresponding  eigenvector \ts $\vb \in \Rb^d_+$\ts.
	Then we get
	\begin{align*}
		\langle \xb+\ep \vb, \bM \. (\xb+\ep \vb) \rangle  \ & = \ \langle \xb, \bM \xb \rangle \ + \ 2 \. \ep \.  \langle \xb, \bM \vb \rangle \. + \. \ep^2 \.  \langle \vb, \bM \vb \rangle \\ &\geq   \  \ep^2 \.  \langle \vb, \bM \vb \rangle
		\ = \ \ep^2 \. \lambda \. \langle \vb, \vb \rangle \ > \ 0,
	\end{align*}
as desired.
	
	Now suppose that all eigenvalues of \. $\bM$ \. are nonpositive.
	Then this implies,  \. $\bM_{\xb \! \xb}=\langle \xb, \bM \xb \rangle < 0$ \.  whenever \. $ \xb \notin \ker \bM$\..
	On the other hand, we have \. $\bM_{\xb \! \xb}=\langle \xb, \bM \xb \rangle \geq 0$ \. by the non-negativity of $\bM$ and $\xb$.
	Combining these two observations, we get  that  \. $\bM\xb=\mathbf{0}$\..
	Analogously, we have \. $\bM \yb= \bM \zb=\mathbf{0}$, which completes the proof.
\end{proof}

\smallskip

The following lemma is a useful consequence of \eqref{eq:Hyp}, and is inspired by the inequality in
\cite[Lemma~7.4.1]{Sch} for mixed volumes.\footnote{Our original proof was more complicated.
		The simplified proof below was suggested to us by Ramon van Handel (Nov.~2022, personal communication). }

\smallskip

\begin{lemma}\label{lem:atlas-corr-quart}	
	Let \. $\bMr$ \. be a nonnegative symmetric $d \times d$ matrix that satisfies \eqref{eq:Hyp},
and let \. $\xb,\yb,\zb \in \Rb^d_+$ \.  be nonnegative vectors.  Then:
	\begin{align}\label{eq:atlas-corr-quart}
		\big(\bMr_{\yb\!\zb} \, \bMr_{\xb \! \xb} - \bMr_{\xb \! \yb}  \, \bMr_{\xb \! \zb}  \big)^2 \ \leq \ \big(\bMr_{\xb \! \yb}^2 - \bMr_{\xb \! \xb} \, \bMr_{ \yb \! \yb}\big) \, \big(\bMr_{\xb \! \zb}^2 - \bMr_{\xb \! \xb} \, \bMr_{\zb \! \zb}\big).
	\end{align}
\end{lemma}

\smallskip

\begin{proof}
	Fix \. $\xb \in \Rb_+^d$ \. and let \. $\aQ: \Rb^d \times \Rb^d \to \Rb$ \.
	be the bilinear form given by
	\[ \aQ(\yb,\zb)  \ := \   \bM_{\yb\!\zb} \, \bM_{\xb \! \xb} - \bM_{\xb \! \yb}  \, \bM_{\xb \! \zb}\..\]
	Then
	 \eqref{eq:Hyp} implies that
	\ts  $\aQ(\yb,\yb)$ \ts is a nonnegative quadratic form.
	It then follows from the Cauchy--Schwarz inequality that
	\[ \aQ(\yb,\zb)^2 \ \leq \ \aQ(\yb,\yb) \. \aQ(\zb,\zb)\ts.  \]
	This is equivalent to \eqref{eq:atlas-corr-quart}, which completes the proof.
\end{proof}

\smallskip

\subsection{The implications}
The following lemma is a simple consequence of
Lemma~\ref{lem:atlas-corr-quart}.

\smallskip

\begin{lemma}\label{lem:atlas-corr-tri}
	Let \. $\bMr$ \. be a nonnegative symmetric $d \times d$ matrix that satisfies \eqref{eq:Hyp}, and let \. $\xb,\yb,\zb \in \Rb^d_+$  \.be nonnegative vectors. Then:
	\begin{align}\label{eq:atlas-corr-tri}
		\bMr_{\yb \! \yb} \. \bMr_{\xb \! \zb}^2  \ + \
		\bMr_{\zb \! \zb} \. \bMr_{\xb \! \yb}^2  \ \leq \ 2 \ts \bMr_{\xb \! \yb} \. \bMr_{\xb \! \zb} \. \bMr_{\yb \! \zb}\..
	\end{align}
\end{lemma}

\smallskip

\begin{proof}
	The lemma clearly follows if \. $\bM \. {\xb}= \bM \. {\yb}=  \bM \.{\zb} =\mathbf{0}$,
	so  assume that this is not the case.
	Then, by Lemma~\ref{lem:Pos},
	we can  substitute \. $\xb$ \. with \. $\xb+\ep \vb$ \.  and then takes the limit $\ep \to 0$ if necessary,
	 so we can additionally assume that \. $\bM_{\xb \! \xb} > 0$\..
	Expanding the squares in \eqref{eq:atlas-corr-quart}, we get:
\begin{equation*}
\aligned
	& \bM_{\yb \! \zb}^2 \. \bM_{\xb \! \xb}^2  \. - \. 2 \ts \bM_{\xb \! \xb} \. \bM_{\xb \! \yb} \. \bM_{\xb \! \zb} \. \bM_{\yb \! \zb}  \. + \. \bM_{\xb \! \yb}^2 \. \bM_{\xb \! \zb}^2 \\
	&\qquad \leq  \ \bM_{\xb \! \yb}^2 \. \bM_{\xb \! \zb}^2 \. + \.
		\bM_{\xb \! \xb}^2 \. \bM_{\yb \! \yb} \. \bM_{\zb \! \zb} \. - \. \bM_{\xb \! \xb} \. \bM_{\yb \! \yb} \. \bM_{\xb \! \zb}^2  \. - \. \bM_{\xb \! \xb} \. \bM_{\zb \! \zb} \. \bM_{\xb \! \yb}^2\..
\endaligned
\end{equation*}
Cancelling the term \. $\bM_{\xb \! \yb}^2 \. \bM_{\xb \! \zb}^2 $ \. from both sides,
	and then dividing both sides by \. $\bM_{\xb \! \xb}$\., we get:
	\begin{align}\label{eq:3det}
		\bM_{\yb \! \zb}^2 \. \bM_{\xb \! \xb}  \. - \. 2 \ts \bM_{\xb \! \yb} \. \bM_{\xb \! \zb} \. \bM_{\yb \! \zb}
		\ \leq \
		\bM_{\xb \! \xb} \. \bM_{\yb \! \yb} \. \bM_{\zb \! \zb} \. - \.  \bM_{\yb \! \yb} \. \bM_{\xb \! \zb}^2  \. - \.  \bM_{\zb \! \zb} \. \bM_{\xb \! \yb}^2\ts.
	\end{align}
This is equivalent to
	\begin{align*}
		\bM_{\yb \! \yb} \. \bM_{\xb \! \zb}^2  \. + \.  \bM_{\zb \! \zb} \. \bM_{\xb \! \yb}^2 \. + \.  \bM_{\xb \! \xb}\big( \bM_{\yb \! \zb}^2 \. - \. \bM_{\yb \! \yb} \, \bM_{\zb \! \zb}  \big)  \ \leq \ 2 \ts \bM_{\xb \! \yb} \. \bM_{\xb \! \zb} \. \bM_{\yb \! \zb}\..
	\end{align*}
By  \eqref{eq:Hyp}, we have \. $\bM_{\yb \! \zb}^2 \ts \geq \ts \bM_{\yb \! \yb} \,  \bM_{\zb \! \zb}\ts$.  This and the above inequality
imply the result.  \end{proof}

\smallskip

\begin{rem}\label{r:Shephard}
Note that \eqref{eq:atlas-corr-tri} implies \eqref{eq:Hyp} by substitution \ts $\zb\gets \xb$.
The inequality \eqref{eq:3det} can be rewritten symmetrically as:
\begin{equation}\label{eq:Shephard}
 \det \begin{bmatrix}
	\bM_{\xb \! \xb} & \bM_{\xb \! \yb} & \bM_{\xb \! \zb} \\
		\bM_{\xb \! \yb} & \bM_{\yb \! \yb} & \bM_{\yb \! \zb} \\
		\bM_{\xb \! \zb} & \bM_{\yb \! \zb} & \bM_{\zb \! \zb} \\
\end{bmatrix}  \ \geq \ 0\ts.
\end{equation}
It can be viewed as the counterpart of \defng{Shephard's inequality}
for mixed volumes, see e.g.~\cite{RvH}.
\end{rem}

%

\smallskip

\begin{lemma}\label{lem:atlas-corr-half}
		Let \. $\bMr$ \. be a nonnegative symmetric $d \times d$ matrix that satisfies \eqref{eq:Hyp}, and let \. $\xb,\yb,\zb \in \Rb^d_+$  \. be nonnegative vectors. Then:
		\begin{align}\label{eq:atlas-corr-half}
				\bMr_{\zb \! \zb} \. \bMr_{\xb \! \yb} \ \leq \ 2  \.  \bMr_{\xb \! \zb} \. \bMr_{\yb \! \zb}\..
		\end{align}
\end{lemma}

\smallskip
\begin{proof}
	Note that the lemma clearly holds if  \. $\bM_{\xb \! \yb} = 0$\.,
	so we can assume that \. $\bM_{\xb \! \yb} >  0$\..
	The lemma now follows by removing the term \. $\bM_{\yb \! \yb} \. \bM_{\xb \! \zb}^2$ from the LHS of \eqref{eq:atlas-corr-tri} and dividing both sides by  \. $\bM_{\xb \! \yb}$\..
\end{proof}

\smallskip

The upper bound in \eqref{eq:atlas-corr-half} can be improved in some cases.

\smallskip

\begin{lemma}\label{lem:atlas-corr-conditional}
Let \. $\bMr$ \. be a nonnegative symmetric $d \times d$ matrix that satisfies \eqref{eq:Hyp},
and let \. $\xb,\yb,\zb \in \Rb^d_+$ \. be nonnegative vectors.
Then at least two out of these three inequalities hold:
$$
\bMr_{\xb \! \xb} \. \bMr_{\yb \! \zb} \ \leq \ \bMr_{\xb \! \yb} \. \bMr_{\xb \! \zb} \ , \quad
\bMr_{\yb \! \yb} \. \bMr_{\xb \! \zb} \ \leq \ \bMr_{\xb \! \yb} \. \bMr_{\yb \! \zb} \quad \text{or}
\quad \bMr_{\zb \! \zb} \. \bMr_{\xb \! \yb} \ \leq \ \bMr_{\xb \! \zb} \. \bMr_{\yb \! \zb} \..
$$
\end{lemma}

\smallskip

\begin{proof}
	Suppose to the contrary that at least two of these inequalities are false.
	Without loss of generality we assume the two inequalities are
	\[ 	\bM_{\xb \! \xb} \. \bM_{\yb \! \zb} \  > \ \bM_{\xb \! \yb} \. \bM_{\xb \! \zb} \qquad \text{ and } \qquad  \bM_{\yb \! \yb} \. \bM_{\xb \! \zb}  \ > \ \bM_{\xb \! \yb} \. \bM_{\yb \! \zb}\.. \]
	Taking the product of these two inequalities gives \. $\bM_{\xb \! \xb} \.  \bM_{\yb \! \yb} \. > \. \bM_{\xb\!\yb}^2.$
This contradicts \eqref{eq:Hyp} and completes the proof.
\end{proof}

\smallskip

\begin{rem}
In conditions of Lemma~\ref{lem:atlas-corr-conditional}, there are cases such that
exactly two out of these three inequalities hold.  For example, the lower
bound in \eqref{eq:poset-corr-tight} is an example in which  \.
$\bM_{\zb \! \zb} \. \bM_{\xb \! \yb} \.  > \. \bM_{\xb \! \zb} \. \bM_{\yb \! \zb}$\.,
with  the choice of  \ts $\xb,\yb,\zb$ \ts as in the proof of
Theorem~\ref{t:poset-corr-deletion-strong}.
%
\end{rem}

\bigskip

\section{Proof of correlation inequalities}\label{s:proof}

In this section we prove the inequalities in the introduction.
The proofs follow the approach in~\cite{CP1}, and are based on the observation
that certain matrices associated with posets satisfy \eqref{eq:Hyp}.

\subsection{The setup}\label{ss:proof-setup}
Fix a poset \ts $P=(X,\prec)$ \ts with \ts $|X|=n$ \ts elements,
and fix an element \ts $\pa \in X$.
For every \ts $k \geq 0$ \ts and \ts $S \subseteq \Ec(P)$,  we write
$$
\aN_k (S)  \, := \,    \big|\{ f \in S \. : \. f(\pa)=k \} \big|\ts .
$$
To simplify the notation, we  write
\[
    \aN_k\big(f(x)=1\big) \quad \text{ to mean } \quad  \aN_k\big(\{ f \in \Ec(P) \. : \. f(x)=1 \} \big).
\]
When the underlying poset $P$ is potentially ambiguous, we will write \. $\aN^{\<P\>}_{k}$ \.
instead of \. $\aN_{k}$\ts.

\smallskip

Let \ts $Z_{\up}$ \ts and \ts $Z_{\down}$ \ts be two distinct copies of $X- \{\pa\}$,
let \ts $Z:=Z_{\up} \, \cup \,  Z_{\down}$\ts, let \ts $d=|Z|=2(n-1)$, and  let \ts $k\in \{2,\ldots,n-1\}$.
We denote by \ts $\bM:= \bM(P,\pa,k)$ \ts
 the symmetric $d \times d$  matrix given by
 \begin{itemize}
 	\item[$\circ$] \, $\aM_{x \ts y} \ := \ \aN_{k+1}\big( f(x)= 1, \ts f(y) =2\big)$  \ if \  $x,y\in\min(Z_{\down})$\ts, \ts $x\ne y$,
 	\vspace*{3 pt}
 	\item[$\circ$] \, $\aM_{x \ts y} \ := \ \aN_{k-1}\big( f(x)= n, f(y) =n-1\big)$  \ if \  $x,y\in\max(Z_{\up})$\ts, \ts $x\ne y$,
 	  	\vspace*{3 pt}
 	\item[$\circ$] \, $\aM_{x \ts y} \ := \ \aN_{k}\big( f(x)= 1, f(y) =n\big)$  \ if \  $x\in\min(Z_{\down})$\ts, $y\in\max(Z_{\up})$\ts,
 	  	\vspace*{3 pt}
 	\item[$\circ$] \, $\aM_{x \ts x} \ := \ \aN_{k+1}\big( f(x)= 1\big) \. - \. \aN_{k+1}\big( f(x)= 2\big)$  \ if \  $x\in\min(Z_{\down})$\ts,
 	   	  	\vspace*{3 pt}
 	\item[$\circ$] \, $\aM_{x \ts x} \ := \ \aN_{k-1}\big( f(x)= n\big) \. - \. \aN_{k-1}\big( f(x)= n-1\big)$  \ if \  $x\in\max(Z_{\up})$\ts,
 	   	 \vspace*{3 pt}
 	\item[$\circ$] \,  $\aM_{x \ts y} \ := \ 0$ \ otherwise.
 	   	 \vspace*{3 pt}
 \end{itemize}

 \smallskip

\begin{prop}[{\cite[Prop~14.9]{CP1}}]\label{prop:M}
The matrix \. $\bMr$ \ts defined above satisfies \eqref{eq:Hyp}.
\end{prop}

\smallskip

We will also use the following combinatorial properties of the matrix $\bM$.
 By a direct calculation, the diagonal entry \ts $\aM_{x \ts x}$ \. for \. $x\in \min(Z_{\down})$, is equal to
\begin{equation}\label{eq:M-diag-comb}
\begin{split}
	\aM_{x  \ts x} \ &= \  \aN_{k+1}\big( f(x)= 1\big) \. - \. \aN_{k+1}\big( f(x)= 2, \ts x \. || \. f^{-1}(1) \big)  \\
	\ &= \  \aN_{k+1}\big( f(x)= 1\big) \. - \. \aN_{k+1}\big( f(x)= 1, \ts x \. || \. f^{-1}(2) \big)  \\
	\ &= \ \aN_{k+1}\big(f(x)=1, \ts x \prec f^{-1}(2) \big),
\end{split}
\end{equation}
where \. $x \,  || \, y$ \.  means that $x$ is incomparable to $y$ in $P$.
By a similar argument, for $x\in \max(Z_{\up})$\ts, we have:
\begin{align}
	\aM_{x \ts x} \ = \   \aN_{k-1}\big(f(x)=n, \ts x \succ f^{-1}(n-1) \big).
\end{align}

\smallskip

\begin{lemma}
	If \. $x\in \min(Z_{\down})$, then:
	 \begin{equation}\label{eq:M-sum-min}
		\sum_{y \ts\in \ts Z_{\down}}  \aMr_{x \ts y} \ = \ \aNr_{k+1}\big(f(x)=1\big) \qquad \text{ and } \qquad \sum_{y \ts\in \ts Z_{\up}}  \aMr_{x \ts y} \ = \ \aNr_{k}\big(f(x)=1\big).
	\end{equation}
	Similarly, if \. $x\in \max(Z_{\up})$, then:
\begin{equation}\label{eq:M-sum-max}
	\sum_{y \ts\in \ts Z_{\down}}  \aMr_{x \ts y} \ = \ \aNr_{k}\big(f(x)=n\big) \qquad \text{ and } \qquad \sum_{y \ts\in \ts Z_{\up}}  \aMr_{x \ts y} \ = \ \aNr_{k-1}\big(f(x)=n\big).
\end{equation}
\end{lemma}

\begin{proof} \ts
	This follows from a direct calculation. The details are straightforward.
\end{proof}

\medskip

\subsection{Correlations for deletion operations}\label{ss:proof-poset-corr-deletion}
Let \ts $P=(X,\prec)$ \ts
be a poset with $n$ elements.  Fix \ts $k \geq 1$ \ts and \ts $\pa\in X$.
As in the introduction, denote by \. $e_k(P):=\aN_{k}^{\<P\>}$ \. the
number of linear extensions \ts $f\in \Ec(P)$ \ts such that \. $f(\pa)=k$.
We start with the following correlation inequality
extending the upper bound in Theorem~\ref{t:poset-corr-deletion}.

\smallskip

\begin{thm}\label{t:poset-corr-deletion-strong}
Let \ts $P=(X,\prec)$ \ts be a poset with \ts $|X|=n>2$ \ts elements.
Fix an element \ts $\pa\in X$ \ts and integer  \ts $1 \leq k\leq n-2$.
Then, for every distinct minimal elements \ts $x,y\in \min(X-\pa)$, we have:
\begin{equation}\label{eq:poset-corr-tight-strong}
e_k(P) \, e_k(P-x-y)  \ \leq \ 2 \. e_k(P-x) \, e_k(P-y).
\end{equation}
\end{thm}

\smallskip

\begin{proof}
Let \. $\xb,\yb \in \Rb^{d}$ \. be the characteristic vectors of \.
$x,y \in \min(Z_{\down})$\ts. Similarly, let \. $\zb \in \Rb^{d}$ \.
be the characteristic vector of \ts $Z_{\up}$\ts.
Let \. $\bM:=\bM(P,\pa,k)$ \. be the matrix in Proposition~\ref{prop:M}.
Note that
	\[
    \bM_{\xb\!\yb} \ = \ \aN_{k+1}\big(f(x)=1, \ts f(y)=2 \big) \ = \  e_{k-1}(P-x-y)
    \]	
and
	\begin{align*}
		\bM_{\xb \! \zb} \ = \   \sum_{z \ts \in \ts Z_{\up}} \aM_{x \ts z} \ =_{\eqref{eq:M-sum-min}} \ \aN_{k}\big(f(x)=1\big) \ = \  e_{k-1}(P-x).
	\end{align*}
We have \. $\bM_{\yb\!\zb} \. = \. e_{k-1}(P-y)$ \. by the same reasoning.
	Finally, note that
	\[  \bM_{\zb\!\zb} \ = \ \sum_{v \ts \in \ts Z_{\up}} \. \sum_{w \ts \in \ts Z_{\up}} \. \aM_{v \ts w} \ =_{\eqref{eq:M-sum-max}} \ e_{k-1}(P).  \]
		Substituting these equations into \eqref{eq:atlas-corr-half}, gives
\[  		e_{k-1}(P) \. e_{k-1}(P-x-y)  \ \leq \ 2 \. e_{k-1}(P- x) \. e_{k-1}(P-y),
\]
		Letting \. $k \gets k+1$ \. implies the result.
\end{proof}

\begin{proof}[Proof of Theorem~\ref{t:poset-corr-deletion}]
Let \. $P':=P+\{\pa\}$ \. be the parallel sum of the poset $P$ and the single element~$\{\pa\}$.
Note that
\begin{alignat*}{2}
 e_{k-1}(P') \ & = \ e(P)\ts, \qquad  & e_{k-1}(P'-x-y)  \ & = \  e(P-x-y)\ts, \\
e_{k-1}(P'-x)  \ & = \  e(P-x)\ts,  \qquad  &e_{k-1}(P'-y)  \ & = \  e(P-y)\ts.
\end{alignat*}
The theorem now follows by applying Theorem~\ref{t:poset-corr-deletion-strong} to the poset~$P'$.
\end{proof}


\medskip

\subsection{Covariances for random linear extensions}\label{ss:proof-poset-covariances}
We are now ready to prove Theorem~\ref{t:poset-covariances-multiple}.

\smallskip

\begin{proof}[Proof of Theorem~\ref{t:poset-covariances-multiple}]
We create a new poset \. $P':=(X',\prec')$ \. as follows.
Let the ground set \. $X':=X\cup\{x,y\}$ \ts be obtained by adding two new elements \. $x,y$.
Let the partial order \. $\prec'$ \. be the closure of the partial order \. $\prec$ \.
and  the extra relations
\begin{equation}\label{eq:poset-covariances-relations}
		x \. \prec' \. v \quad \text{for every } \ v \in A,  \qquad  y \. \prec' \. w \quad \text{for every } \ w \in B.
\end{equation}
	Note that $x$ and $y$ are minimal elements of $P'$.
	
Clearly, \. $e(P'-x-y)  =  e(P)$.
Note that, for every linear extension \ts $f\in \Ec(P)=\Ec(P'-x-y)$,
there are exactly \. $f_{\min}(B)$ \. ways to add $y$ to form a linear extension of~$P'-x$.
This implies that
		\[ e(P'-x) \ = \  \sum_{i\geq 1} \. i \.\cdot \. \aN^{\<P\>}\big(f_{\min}(B)=i\big) \ = \  \Eb\big[ f_{\min}(B)\big] \.\cdot \.  e(P). \]
	By the same reasoning, we have
	\[e(P'-y)  \ = \ \Eb\big[ f_{\min}(A)\big] \.\cdot \. e(P).   \]
		Finally, note that for every linear extension \ts $f\in \Ec(P)=\Ec(P'-x-y)$, there are exactly \. $f_{\min}(A) \. f_{\min}(B) \. + \. \min\{f_{\min}(A), \ts f_{\min}(B)\}$ \. ways to add $x$ and $y$ to form a linear extension of~$P'$.
		This implies that
	\begin{align*}
		e(P') \ &= \  \sum_{i\geq 1}\. \sum_{j\geq 1} \. \big(i\ts j\. + \. \min\{i,j\}\big) \.\cdot \. \aN^{\<P\>}\big(f_{\min}(A)=i, \ts f_{\min}(B)=j\big) \\
		\ &= \ \bigg( \Eb\big[f_{\min}(A) \ts f_{\min}(B)\big] \ + \ \Eb\big[f_{\min}(A \cup B)\big]  \bigg) \. \cdot \. e(P).
	\end{align*}
	Combining all these equations, we get
	\[  \frac{e(P') \.\cdot \. e(P'-x-y)}{e(P'-x) \.\cdot \. e(P'-y)} \ = \
	\frac{\Eb\big[f_{\min}(A) \ts f_{\min}(B) \big] \. + \. \Eb\big[f_{\min}(A\cup B)\big]}{\Eb\big[f_{\min}(A)\big] \ts \Eb\big[f_{\min}(B)\big]} \..
\]
	The result now follows from the upper bound in Theorem~\ref{t:poset-corr-deletion} applied
to the poset~$P'$.
\end{proof}

\medskip

\subsection{Stanley type inequalities for subsets}\label{ss:proof-poset-Stanley-subset}
As in~$\S$\ref{ss:proof-poset-corr-deletion}, we start with the following result
implying Corollary~\ref{c:Stanley-poset-two}.

\smallskip

\begin{lemma}\label{lem:poset-Stanley-subset-cdf-Nform}
Let \ts $P=(X,\prec)$ \ts be a poset on \ts $|X| = n \ge 3$ \ts elements,
let $\pa\in X$, and let \ts $A \subseteq X-a$ \ts be
a nonempty subset.  Then:
	\begin{align*}
     \aNr_k\big(1\notin f(A), \ts 2\in f(A)\big)^2  \ \geq \ \aNr_k\big(1\in f(A), \ts 2 \in f(A\!\uparrow)\big)  \. \cdot \.
		\aNr_k(1,2\notin f(A)),
	\end{align*}
for every \ts $3 \leq k \leq n$.
\end{lemma}

\smallskip

\begin{proof}
Let \. $\xb \in \Rb^{d}$ \. be the characteristic vector of \. $A  \subseteq  Z_{\down}$\., and
let \. $\yb \in \Rb^{d}$ \. be the characteristic vector of \. $(X-A-a) \subseteq  Z_{\down}$\..
	Let \. $\bM:=\bM(P,\pa,k)$ \. be the matrix in Proposition~\ref{prop:M}.
	Let \. $A' \ts := \ts A \cap \min(P)$\..
	Then it follows from \eqref{eq:M-diag-comb} that:
	\begin{align*}
		\bM_{\xb \! \xb} \ &= \ \sum_{x \ts \in \ts A' } \. \Big[ \aN_{k+1}\big(f(x)=1, \ts x \prec f^{-1}(2)\big) \, + \, \sum_{\substack{y \ts \in \ts A' , \. y \. || \. x}}  \aN_{k+1}\big(f(x)=1, \ts f(y)=2\big) \Big].
	\end{align*}
	Since the terms in the sums above  are only nonzero if $x, y \in \min(P)$,
	it then follows that
	\[	\bM_{\xb \! \xb}		\ = \ \sum_{x \ts \in \ts A } \. \Big[ \aN_{k+1}\big(f(x)=1, \ts x \prec f^{-1}(2)\big) \, + \, \sum_{\substack{y \ts \in \ts A, \. y \. || \. x}}  \aN_{k+1}\big(f(x)=1, \ts f(y)=2\big) \Big].\]
	This is then equal to
	\begin{equation}\label{eq:alfa1}
	\aligned	
        \bM_{\xb \! \xb}
       \    = \  \aN_{k+1}\big( 1\in f(A), \ts 2 \in f(A\!\uparrow) \big),
    \endaligned
    \end{equation}
    which follows   from the definition  of
    the upper closure \ts $A\!\uparrow$.


	By the same argument that proves \eqref{eq:alfa1}, we have:
	\begin{equation}\label{eq:alfa2}
		\begin{split}
				 \bM_{\yb \! \yb}
				 \ &= \ \sum_{x \ts \notin \ts A } \. \Big[ \aN_{k+1}\big(f(x)=1, \ts x \prec f^{-1}(2)\big) \, + \, \sum_{\substack{y \ts \notin \ts A, \. y \. || \. x}}  \aN_{k+1}\big(f(x)=1, \ts f(y)=2\big) \Big].\\
				  &\geq  \,  \aN_{k+1} \big(1,2\notin f(A)\big).
		\end{split}
	\end{equation}	

Now, let 	       \. $B':= \min(P)\sm A$\..
Then
	\begin{align*}
		\bM_{\xb\!\yb} \, = \, \bM_{\yb\!\xb} \, = \,  \sum_{y \in B'} \. \sum_{x \in A'} \. \aN_{k+1}\big(f(y)=1, \. f(x)=2\big) \ = \  \aN_{k+1}\big(1\in f(B'), \ts 2\in f(A')\big).
	\end{align*}
	By the definition of $B'$, we then get
	\begin{equation}\label{eq:alfa3}
		\bM_{\xb\!\yb} \  \ = \  \aN_{k+1}\big(1\notin f(A), \ts 2\in f(A')\big) \ \leq  \ \aN_{k+1}\big(1\notin f(A), \ts 2\in f(A)\big).
	\end{equation}
	Combining \eqref{eq:alfa1}, \eqref{eq:alfa2}, \eqref{eq:alfa3} into \eqref{eq:Hyp},
and substituting \ts $k \gets k-1$, we obtain the result.
\end{proof}

\smallskip

\begin{proof}[Proof of the first part of Theorem~\ref{t:poset-Stanley-subset}]
Let \. $P':=P+\{\pa\}$ \. be the parallel sum of the poset $P$ and the single element~$\{\pa\}$.
%
	Note that \. $\aN^{\<P'\>}_{k} = \aN^{\<P\>}$\..
	Applying Lemma~\ref{lem:poset-Stanley-subset-cdf-Nform} to the poset $P'$,
	we get
		\begin{align}\label{eq:Stanley-Lemma-cconsequence}
	\aN \big(1\notin f(A), \ts 2 \in f(A)\big)^2 \ \geq \ 	\aN \big(1\in f(A), \ts 2 \in f(A\!\uparrow) \big)  \. \cdot \.
		\aN\big(1,2\notin f(A)\big)\..
	\end{align}
	Now let  \. $A':=  A \ts \cap \ts \min(P)$ \. and \. $B':= \min(P)\sm A$\.. Note that
\begin{align}\label{eq:Stanley-Lemma-cconsequence-calc}\aligned
\aN \big(1\in f(A), \ts 2 \in f(A\!\uparrow) \big) \, & =  \, \aN\big(1\in f(A) \big) \.
 - \. \aN \big(1\in f(A), \ts 2 \in f(B') \big) \\
 &  =  \, \aN\big(1\in f(A) \big) \.
 - \. \aN \big(1\in f(B'), \ts 2 \in f(A') \big) \\
&  \geq  \, \aN\big(1\in f(A) \big) \.
 - \. \aN \big(1\notin f(A), \ts 2 \in f(A) \big).
\endaligned
 \end{align}
Here the first equality is by  the definition of the upper closure $A\!\uparrow$,
the second equality is obtained by swapping \. $1 \lra 2$ \. in~$f$, and the third inequality is by the same argument that proves \eqref{eq:alfa3}.

Substituting \eqref{eq:Stanley-Lemma-cconsequence-calc} into~\eqref{eq:Stanley-Lemma-cconsequence} and
dividing both sides of the inequality by \. $e(P)^2$, we get a probabilistic inequality:
		\begin{equation}\label{eq:beta1}
	\Pb \big(1\notin f(A), \ts 2 \in f(A)\big)^2 \ \geq \ 	\big[ \Pb \big(1\in f(A) \big)  - \Pb \big(1\notin f(A), \ts 2 \in f(A) \big) \big]  \.\cdot \.
		\Pb\big(1,2\notin f(A)\big)\..
	\end{equation}
Let \. $\al := \Pb\big(1\notin f(A)\big)$, \. $\be := \Pb\big(1,2\notin f(A)\big)$.
 	Note that
 	\[  \Pb \big(1\in f(A) \big) \, =  \, 1-\al, \qquad \Pb \big(1\notin f(A), \ts 2 \in f(A) \big) \, = \,  \al\ts -\ts \be. \]
	Substituting this into \eqref{eq:beta1}, we get
	\[  (\al-\be)^2 \, \geq \, (1-2\al+\be) \. \be\.,  \]
	which is equivalent to \. $\al^2 \geq \be$\..  This gives \eqref{eq:Stanley-subset1}, as desired.
\end{proof}

\smallskip

To finish the proof of  Theorem~\ref{t:poset-Stanley-subset} we need the following lemma.

\smallskip

\begin{lemma}\label{lem:poset-Stanley-subset-yinyang-Nform}
Let \ts $P=(X,\prec)$ \ts be a poset on \ts $|X| = n \ge 3$ \ts elements,
let $\pa\in X$, and let \ts $A \subseteq X-a$ \ts be
a nonempty subset.  Then:
		\begin{align*}
	\aNr_k\big(1\in f(A)\big)^2 \ \geq \ 	\aNr_k \big(1\in f(A), \ts 2\in f(A\!\uparrow)  \big)  \.
		\cdot \. e_k(P),
	\end{align*}
for every \ts $3 \leq k \leq n$.
\end{lemma}

\begin{proof}
Let \. $\xb \in \Rb_+^{d}$ \. be the characteristic vector of \. $A  \subseteq  Z_{\down}$\., and
let \. $\yb \in \Rb_+^{d}$ \. be the characteristic vector of \. $Z_{\down}$\..
Finally, let \. $\bM:=\bM(P,\pa,k)$ \. be the matrix in Proposition~\ref{prop:M}.
It then follows from the same argument as in \eqref{eq:alfa1} that
		\begin{equation*}
		\bM_{\xb\! \xb} \ = \ \aN_{k+1}\big(1\in f(A), \ts 2 \in f(A\!\uparrow)\big).
	\end{equation*}
	Note also that
		\[  \bM_{\yb\!\yb} \ = \ \sum_{v \ts \in \ts Z_{\down}} \. \sum_{w \ts \in \ts  Z_{\down}} \. \aM_{v \ts w} \ =_{\eqref{eq:M-sum-min}} \ e_{k+1}(P).  \]
Now let  \. $A' \ts := \ts A \ts \cap  \ts \min(P)$\., and note that
	\begin{align*}
		\bM_{\xb\!\yb} \ =_{\eqref{eq:M-sum-min}} \ \sum_{x \in A'} \.  \. \aN_{k+1}\big( f(x)=1\big)
 \ = \   \aN_{k+1}\big(1\in f(A')\big) \ = \ \aN_{k+1}\big(1\in f(A)\big).
	\end{align*}
		Combining these three equations with \eqref{eq:Hyp},
and substituting \ts $k \gets k-1$ \ts implies the result.
\end{proof}

\smallskip

\begin{proof}[Proof of the second part of Theorem~\ref{t:poset-Stanley-subset}]
As before, let \. $P':=P+\{\pa\}$ \. be the parallel sum of the poset $P$ and the single element~$\{\pa\}$,
and note that \. $\aN^{\<P'\>}_{k} = \aN^{\<P\>}$.
Applying Lemma~\ref{lem:poset-Stanley-subset-yinyang-Nform} to the poset $P'$, we get
	\begin{align*}
		\aN\big(1\in f(A)\big)^2 \ \geq \ \aN \big(1\in f(A), \ts 2\in f(A\!\uparrow) \big) \.\cdot  \.
		e(P)\..
	\end{align*}
The inequality~\eqref{eq:Stanley-Lemma-cconsequence-calc} gives
	\begin{align*}
 \aN\big(1\in f(A)\big)^2 \ \geq \ 	\Big[ \aN \big(1\in f(A)\big)  - \aN \big(1\notin f(A), \ts 2\in f(A) \big) \Big] \.\cdot  \.
		e(P)\..
	\end{align*}
Dividing both sides of the inequality by $e(P)^2$, we get a probabilistic version:
				\begin{align*}
	\Pb\big(1\in f(A)\big)^2 \ \geq \  \Pb \big(1\in f(A) \big)  - \Pb \big(1\notin f(A), \ts 2 \in f(A) \big).
	\end{align*}
This inequality is equivalent to \eqref{eq:Stanley-subset2}, as desired.
\end{proof}

\medskip

\subsection{Multiple subsets}
\label{ss:proof-poset-disjoint}
Let \ts $P=(X,\prec)$ \ts be a poset on \ts $|X| = n$ \ts elements and fix \ts
$\pa\in X$.  From this point on, let \. $A_1,A_2,A_3\ssu \min(P-a)$ \. be
disjoint subsets of minimal elements.
Denote by \. $\xb_i \in \Rb^{d}$, $1\le i\le 3$,  the characteristic vectors of \.
$A_i  \subseteq  Z_{\down}$\..
%
Let \. $\bM:=\bM(P,\pa,k)$ \. be the matrix in Proposition~\ref{prop:M}.
\smallskip

\begin{lemma}\label{lem:poset-disjoint-comb-interpret}
For all \. $1\le i,j\le 3$, we have:
	\begin{equation}\label{eq:poset-disjoint-comb-interpret}
\bMr_{\xb_i\!\xb_j} \ = \ \left\{ \aligned
& \, \aNr_{k+1} \big(1\in f(A_i), \ts 2 \in f(A_i\!\uparrow) \big) \quad \text{if} \ \ i=j, \\
& \, \aNr_{k+1} \big(1\in f(A_i), \. 2\in f(A_j)  \big) \quad \text{otherwise}.	
\endaligned \right.
\end{equation}
\end{lemma}
%

\smallskip

\begin{proof}  Using \eqref{eq:alfa1} with \. $\xb \gets \xb_i$ \.  and \. $A \gets A_i$ \. we have:
$$
\bM_{\xb_i\!\xb_i} \, = \,  \aN_{k+1} \big(1\in f(A_i), \ts 2 \in f(A_i\!\uparrow) \big).
$$
Similarly, for \ts $i\ne j$, we have:
	\begin{align*}
		\bM_{\xb_i \! \xb_j} \ = \  \sum_{x \in A_i} \. \sum_{y \in A_j}\aN_{k+1}\big(f(x)=1, f(y)=2\big) \ = \ \aN_{k+1} \big(1\in f(A_i), \. 2\in f(A_j)  \big),
	\end{align*}
	which completes the proof.
\end{proof}

\smallskip

\begin{proof}[Proof of Theorem~\ref{t:poset-disjoint-logconcave}]
	By \eqref{eq:Hyp}, we have:
	\[ \bM_{\xb_1\!\xb_2}^2 \ \geq \ \bM_{\xb_1\!\xb_1} \. \bM_{\xb_2\!\xb_2}\.. \]
	Using \eqref{eq:poset-disjoint-comb-interpret} in this inequality, we get:
	\begin{align*}
		\aN_{k+1} \big(1\in f(A_1), \ts 2\in f(A_2)  \big)^2 \ \geq \ \aN_{k+1} \big(1\in f(A_1), \ts 2 \in f(A_1\!\uparrow) \big) \cdot
		\aN_{k+1} \big(1\in f(A_2), \ts 2 \in f(A_2\!\uparrow) \big)\..
	\end{align*}
As before, let \. $P':=P+\{\pa\}$ \. be the parallel sum of the poset $P$ and the single element~$\{\pa\}$,
and note that \. $\aN^{\<P'\>}_{k+1} = \aN^{\<P\>}$.
This and the substitution \. $A_1 \gets A$\., \. $A_2 \gets B$\.,  give
		\begin{equation*}
		\aN \big(1\in f(A), \ts 2\in f(B)  \big)^2 \, \geq \, \aN \big(1\in f(A), \ts 2 \in f(A\!\uparrow) \big) \cdot
		\aN \big(1\in f(B), \ts 2 \in f(B\!\uparrow) \big),
	\end{equation*}
as desired.
\end{proof}

\smallskip

\begin{proof}[Proof of Theorem~\ref{t:poset-disjoint-half-correlation}]
	Substituting \. $\xb \gets \xb_1$\., \. \. $\yb \gets \xb_2$\., \. $\zb \gets \xb_3$ \. and \. $A_1 \gets A$, \.
	\. $A_2 \gets B$, \. $A_3 \gets C$,
	into Lemma~\ref{lem:atlas-corr-half} and Lemma~\ref{lem:poset-disjoint-comb-interpret}, we get
$$
\aligned
& \aN_{k+1} \big(1\in f(C), \ts 2 \in f(C\!\uparrow)\big) \. \cdot \. \aN_{k+1} \big(1\in f(A), \ts 2\in f(B)\big) \\
& \hskip.5cm \ \leq \  2 \. \aN_{k+1} \big(1\in f(A), \ts 2\in f(C)\big) \.\cdot \. \aN_{k+1} \big(1\in f(B), \ts 2\in f(C)\big)\ts.
\endaligned
$$
As before, let \. $P':=P+\{\pa\}$ \. be the parallel sum of the poset $P$ and the single element~$\{\pa\}$,
and repeat the substitution argument \. $\aN_{k+1} \gets \aN$ \. as above.  The first inequality
\eqref{eq:Stanley-ext-12} now follows by dividing both sides of the equation by  \ts $e(P)^2$.
	
The second inequality \eqref{eq:Stanley-ext-12-prime} follows from a similar argument applied to
Lemma~\ref{lem:atlas-corr-conditional}.  We omit the details for brevity.
\end{proof}

\bigskip

\section{Log-concavity implications}\label{s:log-proof}

\subsection{Proof of Corollary~\ref{c:poset-Stanley-cdf}}\label{ss:cor-proof-two-lemmas}
Both inequalities in the corollary follow immediately from the following
two general log-concavity results.  The proofs are straightforward and
included here for completeness.\footnote{For example, a continuous
version of Lemma~\ref{l:poset-Stanley-one} is well-known in the Economics
literature, see \cite[Thm~1]{BB}. }

\smallskip

\begin{lemma}\label{l:poset-Stanley-one}
Let \ts $\{p_1\ts, \ldots \ts, p_n\}$ \ts be a log-concave sequence of
real numbers \ts $p_k \ge 0$.  Let \. $s_k :=  p_k+p_{k+1} + \ldots + p_n$\ts,
for all \. $1 \leq k \leq n$. Then \. $\{s_1,\ldots,s_n\}$ \. is also log-concave.
\end{lemma}

\begin{proof}	After expanding and simplifying the terms, the log-concavity \ts $s_{k-1}\ts s_{k+1} \le s_k^2$ \ts
is equivalent to
	\[ \sum_{i=k}^{n-1} \. p_{k-1} \. p_{i+1}  \,  \leq  \, \sum_{i=k}^n \. p_k \. p_{i}.\]
	By the log-concavity of \ts $\{p_i\}$, we have \ts $p_{k-1} \. p_{i+1} \. \leq  \. p_k \. p_i$ \. for every \ts $i \geq k$.
	Thus the difference between the RHS and LHS of the inequality above is at least $p_{k} \. p_n\ge 0$.
\end{proof}

\smallskip

\begin{lemma}
Let \ts $\{p_1\ts, \ts \ldots \ts, p_n\}$ \ts be a log-concave sequence of
real numbers \ts $p_i \ge 0$. 	Then:
	\[ p_1 \. (p_2 + \ldots +p_{n-1} + p_n) \  \leq  \ p_2 \. (p_1+ \ldots +p_{n-1} +p_n ).\]
\end{lemma}

\begin{proof}
By the log-concavity of \ts $\{p_i\}$, we have \ts  \. $p_{1} \. p_{i} \. \leq  \. p_2 \. p_{i-1}$ \. for every \ts $i \geq 3$.
Thus, the difference between the RHS and LHS of the inequality in the lemma is at least \. $p_{2} \. p_n\ge 0$.
\end{proof}

\subsection{Proof of Proposition~\ref{p:log-Petrov}}\label{ss:cor-proof-log-lemma}
We prove the
proposition by analogy with the lemmas above, as a
polynomial inequality.\footnote{Our original proof of the lemma was more complicated.
The proof below was suggested to us by F\"edor~Petrov (Oct.~2022, personal communication). }

\smallskip

\begin{lemma}  \label{l:log-Petrov}
Let \ts $\{p_1\ts, \ts \ldots \ts, p_n\}$ \ts be a log-concave sequence of
real numbers \ts $p_i \ge 0$. Then:
$$
\big(p_1  \ts  + \ts  p_2 \ts  + \ts \ldots \ts + \ts p_n\big)
\big(1^2 p_1 \ts + \ts 2^2 p_2 \ts  + \ts \ldots \ts + \ts n^2 p_n \big) \. \le \. 2\ts
\big(1 \ts p_1 \ts  + \ts  2 \ts p_2 \ts + \ts \ldots \ts + \ts n \ts p_n \big)^2.
$$
\end{lemma}

\begin{proof}
Write the difference of the RHS and LHS as the sum \. $\al_2 + \al_3 + \ldots + \al_{2n}\ts$
over terms \ts $p_ip_j$ \ts with a fixed sum \ts $i+j$.  We have:
$$
\al_{k+1} \. = \. p_1 \ts p_k \ts \big(4k-1-k^2\big) \. + \.  p_2 \ts p_{k-1} \ts
\big(8(k-1)-2^2-(k-1)^2\big) \. + \. \ldots
$$
Observe that the coefficients in \ts $\al_{k+1}$ \ts start negative and become positive.
The total sum of the coefficients in \ts $\al_{k+1}$ \ts is equal to \.
$\frac{1}{3} \ts k(k+1)(k+2) \. - \. \frac{1}{6} \ts k(k+1)(2k+1) >0$.  By the log-concavity of \ts $\{p_i\}$, we have \ts $p_1p_k \le p_2p_{k-1} \le \ldots \ts$ We conclude that the
positive terms in the sum \ts $\al_{k+1}$ \ts dominate the negative terms, which implies
the result.
\end{proof}

\bigskip

\section{Final remarks} \label{s:finrem}

\subsection{}\label{ss:finrem-Stanley}
The geometric approach via the Alexandrov--Fenchel inequality was introduced
by Stanley \cite{Sta} and further explored in \cite{KS,RvH}.  More recently,
this approach was employed to study the equality conditions \cite{CP-AFequality,MS,SvH,vHYZ}.
The vanishing conditions in Stanley inequality (Theorem~\ref{t:poset-Stanley})
have purely combinatorial proofs, see \cite[Lem.~15.2]{SvH}.  See also generalizations
in \cite[Thm~8.5]{CPP2}, \cite[Thm~1.12]{CPP} and \cite[Thm~5.3]{MS}.  The same holds
for the uniqueness conditions \cite[Thm~7.5]{CPP}.

It remains open whether Stanley inequality \eqref{eq:Stanley-thm} has a combinatorial proof.
Formally, it is not known whether the difference of the RHS and the LHS is in~$\SP$, 
although a negative result of this type was recently obtained in \cite{CP-AFequality} 
for the generalized Stanley's inequality under standard complexity assumptions. 
This problem is similarly open for our new correlation inequalities, since the 
combinatorial atlas
technology is inherently ineffective.  This is in contrast with other poset
inequalities, see~\cite{CPP}.  We refer to the lengthy discussion of
these problems in \cite[$\S$6]{Pak-what}.

\subsection{}\label{ss:finrem-easy}
One can ask if the atlas technology we use largely as a black box in this paper
is really necessary to derive our results?  We certainly believe this to be the
case, since there are so few other tools.  For now, let us offer a word of caution
to the reader accepting the challenge:  while our inequalities can certainly appear
intuitively obvious at least in special cases, this impression is largely deceptive
in the generality of all finite posets.

For example, the proof of the upper bound of~\eqref{eq:poset-corr-tight} given in
Example~\ref{ex:poset-corr-upper-two} for the case of exactly two minimal elements,
is elementary indeed. Thus one might be tempted to extend this proof to general posets.
We hope you succeed, but keep in mind that the only proof we have is as a corollary
of a difficult Theorem~\ref{t:poset-corr-deletion-strong}.

This ``intuitively obvious'' phenomenon is both quite old and completely
understandable.  For example, \defn{Winkler's inequality} \ts \cite[Thm~1]{Win-ave} \ts
certainly appears that way:
\begin{equation}\label{eq:Win}
\Eb\big[f(x)\big] \, \le \, \Eb\big[f(x) \. \big| \. f(x) > f(y)\big] \quad
\text{for every incomparable} \ \ x, y \in X\ts.
\end{equation}
To this day, the only known proof of~\eqref{eq:Win} is the original proof
which uses the XYZ inequality.

\subsection{}\label{ss:finrem-1323}
The famous \. \defng{$\frac{1}{3}$--$\frac{2}{3}$ \. Conjecture} \. states that for every
finite poset \ts $P=(X,\prec)$ \ts there are elements \ts $x,y\in X$, such that
$$
\frac{1}{3} \ \le  \ \Pb\bigl[f(x)<f(y)\bigr] \ \le  \ \frac{2}{3}\,.
$$
This conjecture was stated independently by Kislitsyn (1968) and Fredman (1875). 
and studied in a long series of papers, see e.g.\ \cite[\S13]{CP23-survey} 
for a detailed overview of the literature.

As we mentioned in the introduction, a version of the conjecture with weaker
constants was first proved in \cite{KS} using the extension of Stanley
inequality~\eqref{eq:Stanley}.  See also 
\cite{BFT} for the currently best known bound.

\subsection{}\label{ss:finrem-CPC}
Let \. $P=(X, \prec)$ \. be a finite poset.  Fix distinct elements \. $x,y,z\in X$.
For \ts $k,\ell \geq 1$, denote
\[ \cF(k,\ell) \ := \ \{f \in \Ec(P) \. : \. f(y)-f(x)=k, \.  f(z)-f(y)=\ell \}, \]
and let \. $\aF(k,\ell) \. := \. \big|\cF(k,\ell)\big|$.
The \defng{cross-product conjecture} \ts (CPC)
by Brightwell--Felsner--Trotter \cite[Thm~3.2]{BFT} \ts states that
\begin{equation}\label{eq:CPC} \tag{CPC}
			 \aF(k,\ell) \.\cdot\. \aF(k+1,\ell+1) \ \leq \ \aF(k+1,\ell) \.\cdot\. \aF(k,\ell+1),
	\end{equation}
for every \ts $k,\ell \geq 1$.  The \ts $k=\ell=1$ \ts case was proved in
\cite[Thm~3.2]{BFT}, and the case of width two posets was proved in \cite{CPP1}.
This remains one of the most challenging open problems in the area.  
A constant ratio bound was recently given in~\cite{CPP-quant}.  
As Brightwell--Felsner--Trotter lamented, ``something more powerful seems to be needed''
to prove the general form of~\eqref{eq:CPC}.

\subsection{}\label{ss:finrem-Fish}
Let \ts $P=(X,\prec)$ \ts be a poset, and let \ts $A, B\ssu X$ \ts be upper ideals.
We denote by \. $e(A)$ \. the number of linear extensions of the poset \ts $(A,\prec)$.
 \defn{Fishburn's inequality} \cite{Fish1} states:
\begin{equation}\label{eq:poset-Fish}
	  \frac{|A \cup B|! \. \cdot \. |A \cap B|!}{|A|! \. \cdot \.  |B|!} \
\leq \ \frac{e(A \cup B) \.\cdot\. e(A\cap B)}{e(A) \.\cdot\. e(B)}\,.
\end{equation}
Note that this is a generalization of the lower bound in
\eqref{eq:poset-corr-tight}, obtained by taking \ts $A:=X-x$ \ts and \ts $B:=X-y$.
In a joint work with Panova, we recently rederive this special case in a
remark \cite[$\S$9.8]{CPP}.  Both the original proof in \cite[$\S$2]{Fish1}
and our proof use the FKG inequality.  We also note that Fishburn's inequality 
\eqref{eq:poset-Fish} is a special case of an inequality of Shepp~\cite[Thm~2]{She},
see e.g. \cite[Lemma~10]{Bri} for this derivation.

\subsection{}\label{ss:finrem-Bjo}
In the context of $\S$\ref{ss:app-Young}, Fishburn's inequality
for straight shapes was studied by Bj\"{o}rner \cite[$\S$6]{Bjo11},
as the following correlation inequality:
\begin{equation}\label{eq:Bjor}
|\SYT(\mu)| \. \cdot \. |\SYT(\nu)| \, \le \,
|\SYT(\mu\vee \nu)| \. \cdot \. |\SYT(\mu\wedge \nu)|\.,
\end{equation}
%
%
where $\vee$ and $\wedge$ refer to the union and intersection of the corresponding Young diagrams.

It was also pointed out it in \cite[$\S$7.4]{Pak-what}, that \eqref{eq:Bjor} is
also a corollary of the very general \defn{Lam--Pylyavskyy inequality}
\cite[Thm~4.5]{LP07}
which applies to skew shapes:
\begin{equation}\label{eq:Bjo-LPP}
|\SYT(\mu/\al)| \. \cdot \. |\SYT(\nu/\be)| \, \le \,
\big|\SYT\big(\mu\vee \nu \. / \.\al\vee \be\big)\big| \. \cdot \. \big|\SYT\big(\mu\wedge \nu \. / \. \al\wedge \be\big)\big|\ts.
\end{equation}
%
This result was reproved and further extended in \cite{LPP07} by an algebraic argument.
In a followup paper \cite{CP3}, we show how  \eqref{eq:Bjo-LPP} and its generalizations
can be proved via the Ahlswede--Daykin inequality~\cite{AD}.

\subsection{}\label{ss:finrem-HW}
In the context of $\S$\ref{ss:SYT-hook-walk}, recall that there is no analogue of the
hook walk for skew Young diagrams, or in fact any direct combinatorial way to sample
from \ts $\SYT(\la/\mu)$, see e.g.\ \cite[$\S$10.3]{MPP1}.
It would be interesting to find a direct combinatorial proof \eqref{eq:poset-corr-tight}
in this case.

\subsection{}\label{ss:finrem-reverse}
In conditions of Theorem~\ref{t:poset-corr-deletion}, when \ts $x\in \min(P)$ \ts
and \ts $y\in \max(P)$, the lower bound in \eqref{eq:poset-corr-tight} reverses
direction:
\begin{equation}\label{eq:poset-corr-reverse}
	 \frac{e(P) \.\cdot \. e(P-x-y)}{e(P-x) \.\cdot \. e(P-y)}  \ \leq \ \frac{n}{n-1}\.,
\end{equation}
see \cite[$\S$9.8]{CPP}.  It would be interesting to find a lower bound for the
LHS of~\eqref{eq:poset-corr-reverse} similar to the RHS of \eqref{eq:poset-corr-tight}.
For example, let \ts $P=C_{n-2}+C_2$ \ts be a parallel sum (disjoint union) of two chains,
and let \ts $C_2=\{x,y \. : \. x \prec y\}$. Then the LHS
of~\eqref{eq:poset-corr-reverse} is equal to \. $\frac{n}{2(n-1)}$\..
We conjecture that this lower bound is tight:

\begin{conj}\label{conj:corr-reverse}  Let \ts $x\in \min(P)$ \ts
and \ts $y\in \max(P)$.  We have:
\begin{equation*}
	 \frac{e(P) \.\cdot \. e(P-x-y)}{e(P-x) \.\cdot \. e(P-y)}  \ \geq \ \frac{1}{2}\..
\end{equation*}
\end{conj}

\subsection{}\label{ss:finrem-equality}
The combinatorial atlas technology does prove the equality conditions
in some cases, such as for Stanley inequality \eqref{eq:Stanley-thm},
see \cite[$\S$1.18]{CP1}.  While this paper does not explore the equality
conditions for any of our correlation inequalities, we believe this is an
important direction to pursue.  Finding equality conditions for both
inequalities in Theorem~\ref{t:poset-corr-deletion} would be especially interesting,
as would be the equality conditions for Corollary~\ref{c:Stanley-poset-two} and for
Theorem~\ref{t:poset-disjoint-logconcave}.

\subsection{}\label{ss:finrem-conj}
Conjectures~\ref{conj:poset-Stanley-subset} and~\ref{conj:corr-reverse} 
are very speculative.  It would be interesting to check them computationally
for sufficiently large posets.

\vskip.6cm

\subsection*{Acknowledgements}
We are grateful to Pasha Galashin, Nikita Gladkov, Alejandro Morales, Greta Panova,
F\"edor Petrov, Michael Saks, Yair Shenfeld  and Ramon van Handel for helpful 
discussions and remarks on the subject.  We are especially thankful to Max Aires 
and Jeff Kahn for showing us Example~\ref{ex:poset-AK}.  The first author was 
partially supported by the Simons Foundation.  Both authors were partially 
supported by the~NSF.

\vskip.9cm


\bigskip

\end{document}